\documentclass[11pt]{article}
\usepackage{pb-diagram}
\usepackage{amsmath,amsthm,amsfonts,amssymb,latexsym}
\usepackage{amsmath}

\usepackage{a4wide}
\usepackage[all]{xy}

\newtheorem{thm}{Theorem}[section]
\newtheorem{cor}[thm]{Corollary}
\newtheorem{lem}[thm]{Lemma}
\newtheorem{prop}[thm]{Proposition}
\newtheorem{ex}[thm]{Example}
\newtheorem{defn}[thm]{Definition}

\newtheorem{rmk}[thm]{Remark}

\newtheorem{theorem}{Theorem}[section]
\newtheorem{definition}[theorem]{Definition}
\newtheorem{lemma}[theorem]{Lemma}
\newtheorem{remark}[theorem]{Remark}
\newtheorem{proposition}[theorem]{Proposition}
\newtheorem{corollary}[theorem]{Corollary}

\renewcommand{\theequation}{\arabic{section}.\arabic{equation}}

\def\a{\alpha}
\def\b{\beta}

\begin{document}

\title{Structures of BiHom-Poisson algebras}
\author{ Sylvain Attan\thanks{D\'{e}partement de Math\'{e}matiques, Universit\'{e} d'Abomey-Calavi
01 BP 4521, Cotonou 01, B\'{e}nin. E.mail: syltane2010@yahoo.fr}\and Ismail Laraiedh\thanks{Departement of Mathematics, Faculty of Sciences, Sfax University, BP 1171, 3000 Sfax, Tunisia. E.mail:
Ismail.laraiedh@gmail.com and Departement of Mathematics, College of Sciences and Humanities - Kowaiyia, Shaqra University,
Kingdom of Saudi Arabia. E.mail:
ismail.laraiedh@su.edu.sa}}

\maketitle

\begin{abstract}
This paper  gives some constructions results and examples of
BiHom-Poisson algebras. Next, BiHom-flexible algebras are defined and it is shown that admissible BiHom-Poisson algebras are BiHom-flexible.
Furthermore, generalized derivations of Bihom-Poisson algebras are introduced and some their basic properties  are given. Finally,  BiHom-Poisson modules and  several constructions of these notions are obtained.
\end{abstract}
BiHom-Poisson modules.
\section{Introduction}
A Poisson algebra $(P,\{\cdot,\cdot\}, \mu)$ consists of a commutative associative algebra
$(A,\mu)$\\ together with a Lie structure $\{\cdot,\cdot\},$ satisfying the Leibniz identity:
$$\{\mu(x,y),z\}=\mu(\{x,z\},y)+\mu(x,\{y,z\})$$
 These algebras firstly appeared in the work of Sim\'eon-Denis Poisson two centuries ago when he was studying the three-body problem in celestial mechanics. Since then, Poisson algebras have shown to be connected to many areas of mathematics and physics. Indeed,
 in mathematics, Poisson algebras play a fundamental role in Poisson geometry \cite{iv}, quantum groups \cite{vca},\cite{vgd} and deformation of commutative associative algebras \cite{mg} whereas in physics, Poisson algebras represent a major part of deformation quantization \cite{mk}, Hamiltonian mechanics \cite{via} and topological field theories \cite{pct}. Poisson-like structures are also used in the study of vertex operator algebras \cite{efd}.

Algebras where the identities defining the structure are twisted by a homomorphism are called Hom-algebras.
Hom-type algebras appeared in the Physics literature of the 1990's, when looking for quantum deformations
of some algebras of vector fields, like Witt and Virasoro algebras
(\cite{AizawaSaito}, \cite{Hu}). It was observed that algebras obtained by deforming certain Lie algebras
no longer satisfied the Jacobi identity, but a
modified version of it involving a homomorphism. An axiomatization of this type of algebras
(called Hom-Lie algebras) was given in \cite{JDS}, \cite{DS}. The associative counterpart of Hom-Lie algebras
(called Hom-associative algebras) has been introduced in \cite{ms}, where it was proved also
that the commutator bracket defined by the multiplication in a Hom-associative
algebra gives rise to a Hom-Lie algebra.

A BiHom-algebra is an algebra in such a way that the identities defining the structure are twisted by two homomorphisms
$\alpha$ and $\beta$. This class of algebras was introduced from a categorical approach in \cite{Gra-Mak-Men-Pan} as an extension of the class of Hom-algebras.
These algebraic structures include BiHom-associative algebras, BiHom-Jordan algebras, BiHom-alternative algebras and BiHom-Lie algebras. 
BiHom-Poisson algebras was first introduced in \cite{XLi} and  studied in 
\cite{hadimi} where in partiular, the concept of module of BiHom-Poisson algebra is introduced and  admissible BiHom-Poisson algebras, and only these BiHom-algebras, are shown to give rise to BiHom-Poisson algebras via polarization.
More applications of BiHom-type can be found in \cite{S S, Guo1, chtioui1,chtioui2, ismail, Li C, Liu2}.

The purpose of this paper is further to study BiHom-Poisson algebras. 
The paper is organized as follows. Section 2 contains some important basic notions and notations related to BiHom-algebras, BiHom-Poisson algebras
and modules over BiHom-associative algebras.
Section 3 presents some constructions results of BiHom-Poisson algebras, BiHom-flexible structures and admissible BiHom-
Poisson algebras. In section 4 we give some basic properties concerning derivation algebras, quasiderivation
algebras and generalized derivation algebras of Bihom-Poisson algebras. In section 5, we introduce and give some properties of BiHom-Poisson modules. Next, we prove that from a given
BiHom-Poisson modules, a sequence of this kind of modules can be constructed and we then define the semi-direct product of BiHom-Poisson algebras. Finally,   we define the first and second cohomology spaces of BiHom-Poisson algebras.

Throughout this paper $\mathbb{K}$ is an algebraically closed field of characteristic $0$ and $A$ is a $\mathbb{K}$-vector space.
\section{Preliminaries}

This section contains necessary important basic notions and notations which will be used in next sections.
For the map $\mu: A^{\otimes 2}\longrightarrow A,$ we will sometimes $\mu(a\otimes b)$ as $\mu(a,b)$ or $ab$  for $a,b\in A$ and if $V$ is another vector space, $\tau_1: A\otimes V\longrightarrow V\otimes A$ (resp. $\tau_2: V\otimes A\longrightarrow A\otimes V$) denote the twist isomorphism $\tau_1(a\otimes v)=v\otimes a$ (resp.
$\tau_2(v\otimes a)=a\otimes v$).
\begin{defn}
A BiHom-module is a pair $(M,\alpha_M,\beta_M)$ consisting of a $\mathbb{K}$-module $M$ and
a linear self-maps $\alpha_M, \beta_M: M\longrightarrow M$  such that $\alpha_M\beta_M=\beta_M\alpha_M.$ A morphism\\
$f: (M,\alpha_M, \beta_M)\rightarrow (N,\alpha_N,\beta_N)$ of BiHom-modules is a linear map
 $f: M\longrightarrow N$ such that $f\alpha_M=\alpha_N f$ and
 $f\beta_M=\beta_N f.$
\end{defn}
\begin{defn}\cite{GRAZIANI}
A BiHom-algebra is a quadruple $(A,\mu,\alpha,\beta)$ in which $(A,\alpha,\beta)$ is a BiHom-module, $\mu : A^{\otimes 2} \rightarrow A$ is a linear map.
The BiHom-algebra $(A,\mu,\alpha;\beta)$ is said to be  multiplicative if $\alpha\circ\mu=\mu\circ\alpha^{\otimes 2}$ and $\beta\circ\mu=\mu\circ\beta^{\otimes 2}$ (multiplicativity).
\end{defn}
\begin{defn}
\begin{enumerate}
\item A BiHom-algebra $(A,\mu,\alpha,\beta)$ is said to be BiHom-commutative if
$$\mu(\beta(x),\alpha(y))=\mu(\beta(y),\alpha(x)),~\forall x,y\in A$$.
\item A BiHom-associative algebra \cite{GRAZIANI} is a multiplicative Bihom-algebra $(A,\mu,\alpha,\beta)$
satisfying the following BiHom-associativity condition:
\begin{equation}\label{BiHom ass identity}
    as_{A}(x,y,z):=\mu(\mu(x,y),\beta(z))-\mu(\alpha(x),\mu(y,z))=0,\ \text{for all}\ x,y,z \in A.\
\end{equation}
\item A BiHom-Lie algebra \cite{GRAZIANI} is a multiplicative Bihom-algebra $(A,\{\cdot,\cdot\},\alpha,\beta)$ satisfying the  BiHom-skew-symmetry and the BiHom-Jacobi identities i.e.
\begin{eqnarray}
&&\{\beta(x),\alpha(y)\}=-\{\beta(y),\alpha(x)\}  \nonumber\\
&&\circlearrowleft_{x,w,z}\{\beta^2(x),\{\beta(y),\alpha(z)\}\}=0 \label{BiHJacob}
\end{eqnarray}
where $\circlearrowleft_{x,y,z}$ denotes the summation over the cyclic permutation on $x,y,z.$
\end{enumerate}
\end{defn}
Clearly, a Hom-associative algebra $(A,\mu,\alpha)$ can be regarded as a BiHom-associative algebra $(A,\mu,\alpha,\alpha)$.
\begin{defn}\label{si}\cite{ XLi}
A BiHom-Poisson algebra consists of a  vector space $A,$ two  bilinear maps
$\mu, \{\cdot,\cdot\}: A^{\otimes2}\longrightarrow A,$  linear maps $\alpha,\beta: A\longrightarrow A$
 such that
\begin{enumerate}
\item
$(A,\mu,\alpha, \beta)$ is a BiHom-commutative BiHom-associative algebra,
\item
$(A,\{\cdot,\cdot\},\alpha, \beta)$ is a BiHom-Lie algebra,
\item
the BiHom-Leibniz identity
\begin{eqnarray}
\{\alpha\beta(x),\mu(y,z)\}=\mu(\{\beta(x),y\},\beta(z))+\mu(\beta(y),\{\alpha(x),z\})\label{BHL}
\end{eqnarray}
is satisfied for all $x,y,z\in A.$
\end{enumerate}
\end{defn}
  In a BiHom-Poisson algebra $(A,\{\cdot,\cdot\},\mu,\alpha,\beta)$, the operations $\mu$ and $\{\cdot,\cdot\}$ are called the
  BiHom-associative product and the BiHom-Poisson bracket, respectively.
\begin{rmk}
  A non-BiHom-commutative BiHom-Poisson algebra is a BiHom-Poisson algebra without the BiHom-commutativity assumption \cite{hadimi}. These Bihom-algebras are called BiHom-Poisson algebras \cite{XLi}.
  \end{rmk}
  \begin{defn}
Let $(A,\{\cdot,\cdot\},\mu,\alpha,\beta)$ a BiHom-Poisson algebra. A subspace $H$ of $A$ is called
\begin{enumerate}
\item
A BiHom-subalgebra of $A$ if \[\alpha(H)\subseteq H,~\beta(H)\subseteq H,~\mu(H,H)\subseteq H ~\textrm{and} ~\{H,H\}\subseteq H.\]
\item
A left-BiHom ideal of $A$ if \[\alpha(H)\subseteq H,~\beta(H)\subseteq H,~\mu(A,H)\subseteq H~\textrm{and} ~\{A,H\}\subseteq H.\]
\item
A right-BiHom ideal of $A$ if \[\alpha(H)\subseteq H,~\beta(H)\subseteq H,~\mu(H,A)\subseteq H~\textrm{and} ~\{H,A\}\subseteq H.\]
\item
A two sided BiHom-ideal if $H$ is both a left and a right BiHom-ideal of $A$.
  \end{enumerate}
  Note that, if $\alpha$ and $\beta$ are bijective, then the notion of left-BiHom ideals is equivalent to the one of right-BiHom ideals.
\begin{defn}
Let $(A,\{\cdot,\cdot\}, \mu, \alpha,\beta)$ be a Bihom-Poisson algebra. If $Z(A) = \{x \in A |~ \{x,y\}=\mu(x, y) = 0,~~\forall y \in A\}$, then $Z(A)$ is called the centralizer of $A$.
\end{defn}
  \end{defn}
\begin{defn}
\label{def:morphism}
Let $(A,\{\cdot,\cdot\},\mu,\alpha,\beta)$ a BiHom-Poisson algebra.
\begin{enumerate}
\item
$A$ is multiplicative if
\[
\alpha\{\cdot,\cdot\} = \{\cdot,\cdot\} \alpha^{\otimes 2},~~\beta\{\cdot,\cdot\} = \{\cdot,\cdot\} \beta^{\otimes 2} \quad\text{and}\quad \alpha\mu = \mu\alpha^{\otimes 2},~~\beta\mu = \mu\beta^{\otimes 2}.
\]
  \item $(A,\mu,\alpha,\beta)$ is said to be regular if $\alpha$ and $\beta$ are algebra automorphisms.
  \item $(A,\mu,\alpha,\beta)$ is said to be involutive if $\alpha$ and $\beta$ are two involutions,  that is $\alpha^2=\beta^2=id$.
\item
Let $(A',\{\cdot,\cdot\}',\mu',\alpha',\beta')$ be another BiHom-Poisson algebra. A weak morphism $f \colon A \to A'$ is a linear map such that
\[
f\{\cdot,\cdot\} = \{\cdot,\cdot\}' f^{\otimes 2} \quad\text{and}\quad
f\mu = \mu'f^{\otimes 2}.
\]
A morphism $f \colon A \to A'$ is a weak morphism such that $f\alpha = \alpha'f$ and $f\beta=\beta' f$.
\end{enumerate}
\end{defn}
Note that a $5$-tuple $(A,\{\cdot,\cdot\},\mu,\alpha,\beta)$ is multiplicative if and only if the twisting map $\alpha,\beta \colon A \to A$ are morphisms.

Denote by $\Gamma_{f}=\{x+f(x);~~x\in A\}\subset A\oplus A'$  the graph of a linear map $f:A\longrightarrow A^{'}$.
\begin{defn} Let $(A,\mu,\alpha,\beta)$ be any BiHom-algebra.
\begin{enumerate}
\item A BiHom-module $(V,\phi,\psi)$ is called an $A$-bimodule if it comes equipped with a left and a right structures maps on $V$ that is morphisms
$\rho_{l}: (A\otimes V, \alpha\otimes\phi,\beta\otimes\psi)\rightarrow (V,\phi,\psi),~a\otimes v\mapsto a.v$ and $\rho_{r}:(V\otimes A,\phi\otimes\alpha,\psi\otimes\psi)\rightarrow (V,\phi,\psi),~v\otimes a\mapsto v.a$ of Bihom-modules.
\item
A morphism $f:(V,\phi,\psi,\rho_{l},\rho_{r})\rightarrow (W,\phi',\psi',\rho_{l}',\rho_{r}')$ of $A$-bimodules is a morphism of the underlying BiHom-modules such that
$$f\circ\rho_{l}=\rho_{l}'\circ(Id_{A}\otimes f)~~and~~f\circ\rho_{r}=\rho_{r}'\circ( f \otimes Id_{A}).$$
That yields the commutative diagrams
  $$
\xymatrix{
 A\otimes V \ar[d]_{  Id_{A}\otimes f}\ar[rr]^{\rho_l}
                && V  \ar[d]^{f}  \\
 A\otimes W \ar[rr]^{\rho_l'}
                && W}\quad\xymatrix{
 V\otimes A \ar[d]_{ f\otimes Id_{A} }\ar[rr]^{\rho_r}
                && V  \ar[d]^{f}  \\
 W\otimes A \ar[rr]^{\rho_r'}
                && W  }
$$

\end{enumerate}
\end{defn}

Now, let consider the following notions for BiHom-associative algebras.
\begin{defn}
Let $(A,\mu,\alpha,\beta)$ be a BiHom-associative algebra, 
$(L,\{\cdot,\cdot\},\alpha,\beta)$ be a Hom-Lie algebra
 and $(V, \phi,\psi)$ be a BiHom-module. Then
\begin{enumerate}
\item A left BiHom-associative $A$-module structure on $V$ \cite{GRAZIANI} consists of a morphism $\rho_{l}: A\otimes V\longrightarrow V$ of BiHom-modules, such that
\begin{eqnarray}\label{00}
\rho_{l}\circ(\alpha\otimes\rho_{l})=\rho_{l}\circ(\mu\otimes\psi).\label{LeftAssMod}
\end{eqnarray}
In terms of diagram, we have
  $$
\xymatrix{
 A\otimes A\otimes V \ar[d]_{ \mu\otimes\psi}\ar[rr]^{\alpha\otimes\rho_{l}}
                && A\otimes V  \ar[d]^{\rho_{l}}  \\
 A\otimes V \ar[rr]^{\rho_{l}}
                && V  }$$
\item A right BiHom-associative $A$-module structure on $V$ \cite{GRAZIANI} consists of a morphism $\rho_{r}: V\otimes A\longrightarrow V$ of BiHom-modules, such that
\begin{eqnarray}\label{11}
\rho_{r}\circ(\phi\otimes\mu)=\rho_{r}\circ(\rho_{r}\otimes\beta).\label{RightAssMod}
\end{eqnarray}
In terms of diagram, we have
 $$
\xymatrix{
V\otimes A\otimes A  \ar[d]_{ \phi\otimes\mu}\ar[rr]^{\rho_{r}\otimes\beta}
                && V\otimes A  \ar[d]^{\rho_{r}}  \\
 V\otimes A \ar[rr]^{\rho_{r}}
                && V  }$$
\item A left BiHom-Lie $L$-module structure on $V$ \cite{hadimi}\cite{cyong} consists of a structure map $\rho: L\otimes V\longrightarrow V$ 
such that 
\begin{eqnarray}
\rho(\{\beta(x),y \},\psi(v))&=&\rho(\alpha\beta(x),\rho(y,v))-\rho(\beta(y),\rho(\alpha(x),v))\label{Li1}
\end{eqnarray}
\end{enumerate}
\end{defn}
\section{BiHom-Poisson algebras}

\subsection{Constructions of BiHom-Poisson algebras}
In this subsection, we provide
some constructions results of BiHom-Poisson algebras.
\begin{prop}
Let $(A,\{\cdot,\cdot\},\mu,\alpha,\beta)$ be a BiHom-Poisson algebra and $I$ be a two-sided BiHom-ideal of $(A,\{\cdot,\cdot\},\mu,\alpha,\beta)$.
Then $(A/I,[\cdot,\cdot],\overline{\mu},\overline{\alpha},\overline{\beta})$ is a BiHom-Poisson algebra where $[\overline{x},\overline{y}] =\overline{\{x,y\}} ,~\overline{\mu}(\overline{x},\overline{y})=\overline{\mu(x,y)}$,
$\overline{\alpha}(\overline{x})=\overline{\alpha(x)}$  and $\overline{\beta}(\overline{x})=\overline{\beta(x)}$, for all $\overline{x},\overline{y}\in A/I$
\begin{proof}
We only prove item 1. of definition \ref{si}, item 2. and item 3. are being proved similarly

For all $\overline{x},\overline{y},\overline{z}\in A/I$ we have
$$\begin{array}{llllll}
as_{A/I}(\overline{x},\overline{y},\overline{z})&=&\overline{\mu}(\overline{\mu}(\overline{x},\overline{y}),\overline{\beta}(\overline{z}))-\overline{\mu}(\overline{\alpha}(\overline{x}),\overline{\mu}(\overline{y},\overline{z}))\\[0.2cm]
&=&\overline{\mu(\mu(x,y),\beta(z))-\mu(\alpha(x),\mu(y,z))}\\[0.2cm]
&=&\overline{0}~(by~ BiHom-associativity ~of ~A).
\end{array}$$
Then $(A/I,\overline{\mu},\overline{\alpha},\overline{\beta})$ is a BiHom-associative algebra.
$$\begin{array}{llll}
\overline{\mu}(\overline{\beta}(\overline{x}),\overline{\alpha}(\overline{y}))
&=&\overline{\mu(\beta(x),\alpha(y))}\\[0.2cm]&=&\overline{\mu(\beta(y),\alpha(x))}~(by~BiHom-commutativity ~of A).\\[0.2cm]
&=&\overline{\mu}(\overline{\beta}(\overline{y}),\overline{\alpha}(\overline{x}))
\end{array}$$Then $(A/I,\overline{\mu},\overline{\alpha},\overline{\beta})$ is a BiHom-commutative algebra.
\end{proof}

\end{prop}
\begin{prop}\label{adimi}\cite{hadimi}
Let $(A, \mu, \alpha,\beta)$ be a regular BiHom-associative algebra. Then\\
$A^{-} = (A, \{\cdot,\cdot \}, \mu, \alpha,\beta)$
is a regular non-BiHom-commutative BiHom-Poisson algebra, where $\{\cdot,\cdot\}=\mu-\mu\circ(\alpha^{-1}\beta\otimes\alpha\beta^{-1})\circ\tau$,

where for all $x,y\in A,~~\tau(x\otimes y)=y\otimes x$.
\end{prop}
\begin{ex} \label{e1}
Consider the 2-dimensional regular BiHom-associative algebras $(A,\mu,\alpha,\beta)$ with a basis $(e_1,e_2),$ (see \cite{GRAZIANI}) defined by

$$\begin{array}{llllll}
  \alpha(e_1)=e_1,&&\alpha(e_2)=\frac{b(1-a)}{a}e_1+a e_2, \\
 \beta(e_1)=e_1,&& \beta(e_2)=be_1+(1-a) e_2, \\
  \mu(e_1,e_1)=e_1, && \mu(e_1,e_2)= be_1+(1-a) e_2, \\
 \mu(e_2,e_1)=\frac{b(1-a)}{a}e_1+ae_2, && \mu(e_2,e_2)=\frac{b}{a}e_2,
\end{array}$$
where $a$, $b$ are parameters in $\mathbb{K}$, with $a\neq 0,1$.
Using the Proposition \ref{adimi}, the 5-tuple $(A,\{\cdot,\cdot\},\mu,\alpha,\beta)$ is a regular  non-BiHom-commutative BiHom-Poisson algebra where, $\{\cdot,\cdot\}=\mu-\mu\circ(\alpha^{-1}\beta\otimes\alpha\beta^{-1})\circ\tau$ and
$$\begin{array}{llllll}
 \alpha^{-1}(e_1)=e_1,&&\alpha^{-1}(e_2)=\frac{b(a-1)}{a^{2}}e_1+\frac{1}{a} e_2, \\
 \beta^{-1}(e_1)=e_1,&& \beta^{-1}(e_2)=\frac{b}{a-1}e_1+\frac{1}{1-a} e_2.
\end{array}$$

\end{ex}

\begin{prop}\label{bylv}\cite{ XLi}
 Let $(A,\{\cdot,\cdot\}_{A},\mu_A,\alpha_{A},\beta_{A})$ and
$(B,\{\cdot,\cdot\}_{B},\mu_B,\alpha_{B},\beta_{B})$ be two BiHom-Poisson algebras. Then there exists a BiHom-Poissson algebra $(A\oplus B,\{\cdot,\cdot\},\mu, \alpha=\alpha_{A}+\alpha_{B},\beta=\beta_{A}+\beta_{B}),$ where the bilinear maps $\{\cdot,\cdot\},\mu:(A\oplus B)^{\times 2}\longrightarrow (A\oplus B)$ are given by
 $$\begin{array}{llll}
 \mu(a_1+b_1,a_2+b_2)=\mu_A(a_1,a_2)+\mu_B(b_1,b_2),\nonumber\\
 \{a_1+b_1,a_2+b_2\}=\{a_{1},a_{2}\}_{A}+\{b_{1},b_{2}\}_{B}, \forall \ a_1,a_2\in A,\ \forall \ b_1,b_2\in B.

 \end{array}$$
 and the linear maps $\beta=\beta_{A}+\beta_{B},~\alpha=\alpha_{A}+\alpha_{B}: (A\oplus B)\longrightarrow (A\oplus B)$ are given by
$$ \begin{array}{lll}
(\alpha_{A}+\alpha_{B})(a+b)&=& \alpha_{A}(a)+\alpha_{B}(b),\\
(\beta_{A}+\beta_{B})(a+b)&=& \beta_{A}(a)+\beta_{B}(b),~~ \forall \ (a,b)\in A\times B.
 \end{array}$$
 \end{prop}
\begin{proof}
It is easy to see that $(A\oplus B,\mu, \alpha_{A}+\alpha_{B},\beta_{A}+\beta_{B})$ is a BiHom-associative and BiHom-commutative algebra and
$(A\oplus B,\{\cdot,\cdot\}, \alpha_{A}+\alpha_{B},\beta_{A}+\beta_{B})$ is a BiHom-Lie algebra. Then $(A\oplus B,\{\cdot,\cdot\},\mu, \alpha_{A}+\alpha_{B},\beta_{A}+\beta_{B})$ is a BiHom-Poisson algebra.
\end{proof}
\begin{prop}
Let $(A,\{\cdot,\cdot\}_{A},\mu_A,\alpha_{1},\beta_{1})$ and $(B,\{\cdot,\cdot\}_{B},\mu_B,\alpha_{2},\beta_{2})$ be two BiHom-Poisson algebras and $\varphi: A\rightarrow B$ be a linear map. Then
$\varphi$ is a morphism from $(A,\{\cdot,\cdot\}_{A},\mu_A,\alpha_{1},\beta_{1})$ to $(B,\{\cdot,\cdot\}_{B},\mu_B,\alpha_{2},\beta_{2})$  if and only if its graph
$\Gamma_{\varphi}$ is a BiHom-subalgebra of $(A\oplus B, \{\cdot,\cdot\},\mu, \alpha,\beta).$
\end{prop}
\begin{proof}
  Let $\varphi: (A, \mu_A, \alpha_{1},\beta_{1})\longrightarrow (B, \mu_B ,\alpha_{2},\beta_{2})$ be a morphism of BiHom-Poisson algebras.
Then for all $u, v\in A,$
$$ \begin{array}{llllllllll}\mu(u+\varphi(u),v+\varphi(v))=(\mu_A(u,v)+\mu_B(\varphi(u),\varphi(v)))=(\mu_A(u,v)+\varphi(\mu_A(u,v))),\\
\{u+\varphi(u),v+\varphi(v)\}=\{u,v\}_{A}+\{\varphi(u),\varphi(v))\}_{B}=\{u,v\}_{A}+\varphi(\{u,v\}_{A})
\end{array}
$$
Thus the graph $\Gamma_{\varphi}$ is closed under the operations $\mu$ and $\{\cdot,\cdot\}$.
Furthermore since $\varphi\circ\alpha_{1}=\alpha_{2}\circ\varphi$, we have $(\alpha_{1}\oplus\alpha_{2})(u, \varphi(u)) = (\alpha_{1}(u),
\alpha_{2}\circ\varphi(u)) = (\alpha_{1}(u), \varphi\circ\alpha_{1}(u)).$ In the same way, we have $(\beta_{1}\oplus\beta_{2})(u, \varphi(u)) = (\beta_{1}(u),
\beta_{2}\circ\varphi(u)) = (\beta_{1}(u), \varphi\circ\beta_{1}(u)),$
which implies that $\Gamma_{\varphi}$ is closed under $\alpha_{1}\oplus\alpha_{2}$ and $\beta_{1}\oplus\beta_{2}$ Thus $\Gamma_{\varphi}$ is a BiHom-subalgebra of
$(A\oplus B,\{\cdot,\cdot\},\mu,\alpha,\beta).$\\
Conversely, if the graph $\Gamma_{\varphi}\subset A\oplus B$ is a
BiHom-subalgebra of
$(A\oplus B,\{\cdot,\cdot\},\mu,\alpha,\beta)$ then we have
$$\begin{array}{llllll}\mu(u+ \varphi(u)), v+ \varphi(v))=(\mu_A(u, v) + \mu_B(\varphi(u), \varphi(v)) )\in\Gamma_{\varphi},\\
\{u+ \varphi(u)), v+ \varphi(v)\}=\{u, v\}_{A} + \{\varphi(u), \varphi(v)\}_{B}\in\Gamma_{\varphi}

\end{array}
$$
which implies that
$$\begin{array}{llll}\mu_B(\varphi(u), \varphi(v))=\varphi(\mu_A(u, v)),\\
\{\varphi(u), \varphi(v)\}_{B}=\varphi(\{u, v\}_{A}).
\end{array}
$$
Furthermore, $(\alpha_{1}\oplus\alpha_{2})(\Gamma_{\varphi})\subset\Gamma_{\varphi},~(\beta_{1}\oplus\beta_{2})(\Gamma_{\varphi})\subset\Gamma_{\varphi}$ implies
$$(\alpha_{1}\oplus\alpha_{2})(u, \varphi(u))=(\alpha_{1}(u), \alpha_{2}\circ\varphi(u)) \in\Gamma_{\varphi},$$
$$(\beta_{1}\oplus\beta_{2})(u, \varphi(u))=(\beta_{1}(u), \beta_{2}\circ\varphi(u)) \in\Gamma_{\varphi}.$$
which is equivalent to the condition $\alpha_{2}\circ\varphi(u)=\varphi\circ\alpha_{1}(u),$ i.e. $ \alpha_{2}\circ\varphi=\varphi\circ\alpha_{1}$. Similarly, $ \beta_{2}\circ\varphi=\varphi\circ\beta_{1}$. Therefore, $\varphi$ is a
morphism of BiHom-Poisson algebras.
\end{proof}
\begin{thm}
\label{thm:twist}
Let $(A,\{\cdot,\cdot\},\mu,\alpha,\beta)$ be a (non-BiHom-commutative) BiHom-Poisson algebra and $\alpha',\beta' \colon A \to A$ be
endomorphisms of $A$ such that any two of the maps $\alpha,\beta,\alpha',\beta'$ commute. Then
\[
A_{\alpha',\beta'} = (A,\{\cdot,\cdot\}_{\alpha',\beta'} =\{\cdot,\cdot\}\circ(\alpha'\otimes\beta'),\mu _{\alpha',\beta'} = \mu\circ(\alpha'\otimes\beta'),\alpha\alpha',\beta\beta')
\]
is also a (non-BiHom-commutative) BiHom-Poisson algebra. Moreover suppose that\\ $(B,\{\cdot,\cdot\}', \mu', \gamma,\delta)$ is another BiHom-Poisson algebra and $\gamma',\delta'$ be
endomorphisms of $B$ such that any two of the maps $\gamma,\delta,\gamma',\delta'$ commute.
If $f : (A,\{\cdot,\cdot\},\mu, \alpha,\beta)\rightarrow (B,\{\cdot,\cdot\}', \mu', \gamma,\delta)$
is a morphism such that $f\circ\alpha'=\gamma'\circ f$ and $f\circ\beta'=\delta'\circ f$, then $f : A_{\alpha',\beta'}\rightarrow B_{\gamma',\delta'}$ is also a morphism.
\end{thm}

\begin{proof} Let give the proof in BiHom-commutativity case.
We only prove item 1. of Definition \ref{si}, item 2. and item 3. can be proved similarly.

For all $x,y,z\in A$ we have
$$\begin{array}{llll}
&&as_{A_{\alpha',\beta'}}(x,y,z)\\
&=&\mu_{\alpha',\beta'}(\mu_{\alpha',\beta'}(x,y),\beta\beta'(z))-\mu_{\alpha',\beta'}(\alpha\alpha'(x),\mu_{\alpha',\beta'}(y,z))\\
&=&\mu_{\alpha',\beta'}(\mu(\alpha'(x),\beta'(y)),\beta\beta'(z))-\mu_{\alpha',\beta'}(\alpha\alpha'(x),\mu(\alpha'(y),\beta'(z)))\\
&=&\mu(\mu(\alpha'^{2}(x),\alpha'\beta'(y)),\beta\beta'^{2}(z))-\mu(\alpha\alpha'^{2}(x),\mu(\alpha'\beta'(y),\beta'^{2}(z)))\\
&=&as_{A}(\alpha'^{2}(x),\alpha'\beta'(y),\beta'^{2}(z))=0~(BiHom-associativity~condition~of~A).
\end{array}$$
Then $(A,\mu _{\alpha',\beta'},\alpha\alpha',\beta\beta')$ is a BiHom-associative algebra.

Now, for all $x,y\in A$ we have:
$$\begin{array}{llllllll}
\mu_{\alpha',\beta'}(\beta\beta'(x),\alpha\alpha'(y))&=&\mu(\alpha'\beta'\beta(x),\alpha\alpha'\beta'(y))\\
&=&\mu(\beta(\alpha'\beta'(x)),\alpha(\alpha'\beta'(y)))\\
&=&\mu(\beta(\alpha'\beta'(y)),\alpha(\alpha'\beta'(x)))~(BiHom-commutativity~in~A).\\
&=&\mu_{\alpha',\beta'}(\beta\beta'(y),\alpha\alpha'(x)).
\end{array}$$
Then $(A,\mu _{\alpha',\beta'},\alpha\alpha',\beta\beta')$ is a BiHom-commutative algebra.

The second part is proved as follows: $\forall x,y\in A$
$$\begin{array}{llllllll}
f\{x,y\}_{\alpha',\beta'}&=&f\{\alpha'(x),\beta'(y)\}\\
&=&\{f\alpha'(x),f\beta'(y)\}'\\
&=&\{\gamma'f(x),\delta' f(y)\}'\\
&=&\{f(x),f(y)\}'_{\gamma',\delta'}.
\end{array}$$
In the same way we have
$f\mu_{\alpha',\beta'}(x,y)=\mu'_{\gamma',\delta'}(f(x),f(y)).$

This finishes the proof.
\end{proof}
Taking $\alpha'=\alpha^{k},~\beta'=\beta^{l}$, yields the following statement.
\begin{cor}
Let $(A,\{\cdot,\cdot\},\mu,\alpha,\beta)$ be a (non-BiHom-commutative) BiHom-Poisson algebra. Then
\[
A_{\alpha^{k},\beta^{l}} = (A,\{\cdot,\cdot\}_{\alpha^{k},\beta^{l}} =\{\cdot,\cdot\}\circ(\alpha^{k}\otimes\beta^{l}),\mu _{\alpha^{k},\beta^{l}} = \mu\circ(\alpha^{k}\otimes\beta^{l}),\alpha^{k+1},\beta^{l+1})
\]
is also a (non-BiHom-commutative) BiHom-Poisson algebra.
\end{cor}
Taking $\alpha=\beta=id$, yields the following statement.
\begin{cor}
Let $(A,\{\cdot,\cdot\},\mu)$ be a (non-commutative) Poisson algebra and $\alpha,\beta \colon A \to A$ be two
endomorphisms such that $\alpha\beta=\beta\alpha$. Then
\[
A_{\alpha,\beta} = (A,\{\cdot,\cdot\}_{\alpha,\beta} =\{\cdot,\cdot\}\circ(\alpha\otimes\beta),\mu _{\alpha,\beta} = \mu\circ(\alpha\otimes\beta),\alpha,\beta)
\]
is also a (non-BiHom-commutative) BiHom-Poisson algebra.
\end{cor}
\begin{defn}
\label{def:deform}
Let $(A,\{\cdot,\cdot\},\mu)$ be a non-commutative Poisson algebra.
\begin{enumerate}
\item
Given two commuting morphisms $\alpha,\beta \colon A \to A$, the triple $A'_{\alpha,\beta}=(A,\{\cdot,\cdot\}_{\alpha,\beta} =\{\cdot,\cdot\}\circ(\alpha\otimes\beta),\mu _{\alpha,\beta} = \mu\circ(\alpha\otimes\beta))$ is called the $(\alpha,\beta)$-twisting of $A$.  A twisting of $A$ is a $(\alpha,\beta)$-twisting of $A$ for some morphisms $\alpha,\beta \colon A \to A$.
\item
The $(\alpha,\beta)$-twisting $A'_{\alpha,\beta}$ of $A$ is called trivial if
\[
\{\cdot,\cdot\}_{\alpha,\beta}= 0 = \mu_{\alpha,\beta}.
\]
$A'_{\alpha,\beta}$ is called non-trivial if either $\{\cdot,\cdot\} \neq 0$ or $\mu_{\alpha,\beta} \neq 0$.
\item
$A$ is called rigid if every twisting of $A$ is either trivial or isomorphic to $A$.
\end{enumerate}
\end{defn}

\begin{prop}
\label{prop:nonrigidity}
Let $(A,\{\cdot,\cdot\},\mu)$ be a non-commutative Poisson algebra.  Suppose there exists two commuting morphisms $\alpha,\beta \colon A \to A$ such that either:
\begin{enumerate}
\item
$\mu_{\alpha,\beta} = \mu\circ(\alpha\otimes\beta)$ is not BiHom-associative or
\item
$\{\cdot,\cdot\}_{\alpha,\beta}=\{\cdot,\cdot\}\circ(\alpha\otimes\beta)$ does not satisfy the BiHom-Jacobi identity.
\end{enumerate}
Then $A$ is not rigid.
\end{prop}

\begin{proof}
The $(\alpha,\beta)$-twisting $A'_{\alpha,\beta}$ is non-trivial, since otherwise $\mu_{\alpha,\beta}$ would be BiHom-associative and $\{\cdot,\cdot\}_{\alpha,\beta}$ would satisfy the BiHom-Jacobi identity. For the same reason, the $(\alpha,\beta)$-twisting $A'_{\alpha,\beta}$ cannot be isomorphic to $A$.
\end{proof}
\begin{rmk}
\label{ex:linear}
Let $(\mathfrak{g},[\cdot,\cdot])$ be a finite-dimensional Lie algebra, and let $(S(\mathfrak{g}),\mu)$ be its symmetric algebra.  If $\{e_i\}_{i=1}^n$ is a basis of $\mathfrak{g}$, then $S(\mathfrak{g})$ is the polynomial algebra $\mathbb{K}[e_1,\ldots,e_n]$.  Suppose the structure constants for $\mathfrak{g}$ are given by
\[
[e_i,e_j] = \sum_{k=1}^n c_{ij}^k e_k.
\]
Then the symmetric algebra $S(\mathfrak{g})$ becomes a Poisson algebra with the Poisson bracket
\begin{equation}
\label{FG}
\{F,G\} = \frac{1}{2} \sum_{i,j,k=1}^n c_{ij}^k e_k \left(\frac{\partial F}{\partial e_i}\frac{\partial G}{\partial e_j} - \frac{\partial F}{\partial e_j}\frac{\partial G}{\partial e_i}\right)
\end{equation}
for $F,G \in S(\mathfrak{g})$.  This Poisson algebra structure on $S(\mathfrak{g})$ is called the linear Poisson structure.  Note that $\{e_i,e_j\} = [e_i,e_j]$.

\end{rmk}
\begin{ex}[\textbf{$S(\mathfrak{sl}(2))$ is not rigid}]
In this example, we show that the symmetric algebra $(S(\mathfrak{sl}(2)),\mu)$ on the Lie algebra $\mathfrak{sl}(2)$, equipped with the linear Poisson structure \eqref{FG}, is not rigid in the sense of Definition \ref{def:deform}.

The Lie algebra $\mathfrak{sl}(2)$ has a basis $\{e,f,h\}$, with respect to which the Lie bracket is given by
\[
[h,e] = 2e,\quad [h,f] = -2f, \quad [e,f] = h.
\]
To show that $S(\mathfrak{sl}(2)) = \mathbb{K}\langle e,f,h\rangle$ is not rigid, consider the Lie algebra morphisms $\alpha,\beta \colon \mathfrak{sl}(2) \to \mathfrak{sl}(2)$ given by
\[
\begin{array}{llllllllll}
\alpha(e) &=& \lambda e,\quad \alpha(f) &=& \lambda^{-1}f, \quad \alpha(h) &=& h,\\
\beta(e) &=& \gamma e,\quad \beta(f) &=& \gamma^{-1}f, \quad \beta(h) &=& h.
\end{array}
\]
where $\lambda,\gamma \in \mathbb{K}$ is a fixed scalar with $\lambda,\gamma \not= 0,1$.  Denote by $\alpha,\beta \colon S(\mathfrak{sl}(2)) \to S(\mathfrak{sl}(2))$ the extended maps, which is a Poisson algebra morphisms.  By Proposition \ref{prop:nonrigidity}, the Poisson algebra $S(\mathfrak{sl}(2))$ is not rigid if $\mu_{\alpha,\beta} = \mu\circ(\alpha\otimes\beta)$ is not BiHom-associative.  We have
\[
\begin{split}
\mu_{\alpha,\beta}(\mu_{\alpha,\beta}(e,h),h) - \mu_{\alpha,\beta}(e,\mu_{\alpha,\beta}(h,h))
&= \alpha^2(e)\alpha\beta(h)\beta(h) - \alpha(e)\alpha\beta(h)\beta^2(h)\\
&= (\lambda^2 - \lambda)eh^2,
\end{split}
\]
which is not $0$ in the symmetric algebra $S(\mathfrak{sl}(2))$ because $\lambda \not= 0,1$.  Therefore, $\mu_{\alpha,\beta}$ is not BiHom-associative, and the linear Poisson structure on $S(\mathfrak{sl}(2))$ is not rigid.
\qed
\end{ex}
\subsection{BiHom-flexibles structures and admissible BiHom-Poisson algebras}
In this subsection, we introduce BiHom-flexible algebras and prove that 
admissible BiHom-Poisson algebras are BiHom-flexible.
\begin{definition}
Let $ (A, \mu, \alpha,\beta)$ be a BiHom-algebra. Then $A$   is called BiHom-flexible algebra if for any $x, y \in A$
\begin{equation} \label{flex_eq}
\mu(\mu(\beta^{2}(x), \alpha\beta(y)),\beta\alpha^{2}(x))-\mu(\alpha\beta^2(x),\mu(\alpha\beta(y),\alpha^2(x))= 0.
\end{equation}
\end{definition}
\begin{remark}
\begin{enumerate}
\item If $\alpha=\beta=Id,$ then $(A, \mu, \alpha,\beta)$ is reduced to a
flexible algebra $(A, \mu).$
\item If $(A,\mu,\alpha)$ is a Hom-flexible algebra \cite{ms}, then
$(A,\mu,\alpha, \alpha)$ is a BiHom-flexible algebra. Conversely, if
$(A,\mu,\alpha, \beta)$ is a BiHom-flexible algebra such $\alpha$ is injective
and $\alpha=\beta$ then, $(A,\mu,\alpha)$ is a Hom-flexible algebra.
\end{enumerate}
\end{remark}
\begin{lemma} \label{flex_lemma} Let $A = (A, \mu, \alpha,\beta)$ be a BiHom-algebra. The following assertions are equivalent
\begin{enumerate}
\item $A$ is BiHom-flexible.
\item For any $x, y \in A,\quad as_{A}(\beta^{2}(x), \alpha\beta(y), \alpha^{2}(x)) = 0$.
\item For any $x, y, z$
\begin{eqnarray}
\label{homflexible}
as_{A}(\beta^{2}(x),\alpha\beta(y),\alpha^{2}(z)) + as_{A}(\beta^{2}(z),\alpha\beta(y),\alpha^{2}(x)) = 0
\end{eqnarray}
\end{enumerate}
\end{lemma}

\begin{proof}
The equivalence of the first two assertions follows from the definition. The  assertion  $as_{A}(\beta^{2}(x - z), \alpha\beta(y), \alpha^{2}(x - z)) = 0$ holds by definition and it is equivalent to $as_{A}(\beta^{2}(x),\alpha\beta(y),\alpha^{2}(z)) + as_{A}(\beta^{2}(z),\alpha\beta(y),\alpha^{2}(x)) = 0$ by linearity.
\end{proof}
It is easy to prove the following
\begin{prop}
Let $(A,\mu)$ be a flexible algebra, $\alpha,\beta:A\rightarrow A$ be two commuting morphisms.. Then the Bihom-algebra $(A,\mu_{\alpha,\beta}=\mu(\alpha\otimes \beta),\alpha, \beta)$
is BiHom-flexible.
\end{prop}
\begin{proof}
Let $(A,\mu)$ be a flexible algebra, $\alpha$ and $\beta$ be morphisms of
$(A,\mu)$. Then for all $x,y\in A,$
\begin{eqnarray}
&&\mu_{\alpha,\beta}(\mu_{\alpha,\beta}(\beta^{2}(x), \alpha\beta(y)),\beta\alpha^{2}(x))-\mu_{\alpha,\beta}(\alpha\beta^2(x),\mu_{\alpha,\beta}(\alpha\beta(y),\alpha^2(x)))\nonumber\\
&&=\mu_{\alpha,\beta}(\mu(\alpha\beta^{2}(x), \alpha\beta^2(y)),\beta\alpha^{2}(x))-\mu_{\alpha,\beta}(\alpha\beta^2(x),\mu(\alpha^2\beta(y),\beta\alpha^2(x)))\nonumber\\
&&=\mu(\mu(\alpha^2\beta^{2}(x), \alpha^2\beta^2(y)),\beta^2\alpha^{2}(x))-\mu(\alpha^2\beta^2(x),\mu(\alpha^2\beta^2(y),\beta^2\alpha^2(x)))=0.\nonumber
\end{eqnarray}
Hence the Bihom-algebra $(A,\mu_{\alpha,\beta}=\mu(\alpha\otimes \beta),\alpha, \beta)$ is BiHom-flexible.
\end{proof}
\begin{corollary}
Any BiHom-associative algebra is BiHom-flexible.
\end{corollary}
\begin{proposition} \label{flex_Leibniz_Proposition} A BiHom-algebra $A = (A, \mu, \alpha,\beta)$ is BiHom-flexible if and only if
\begin{equation} \label{flex_Leibniz}
\{\alpha\beta(x),\alpha\beta(y) \diamond \alpha^{2}(z)\} - \{\beta^{2}(x),\alpha\beta(y)\} \diamond \alpha^{2}\beta(z) - \alpha\beta^{2}(y) \diamond \{\alpha\beta(x),\alpha^{2}(z)\}=0,
\end{equation}
where $\{\cdot,\cdot\} = \frac{1}{2}(\mu - \mu\circ(\alpha^{-1}\beta\otimes\alpha\beta^{-1})\circ\tau))$ and $\diamond = \frac{1}{2}(\mu + \mu\circ(\alpha^{-1}\beta\otimes\alpha\beta^{-1})\circ\tau)$.
\end{proposition}

\begin{proof}
Since $\{\cdot,\cdot\} = \frac{1}{2}(\mu - \mu\circ(\alpha^{-1}\beta\otimes\alpha\beta^{-1})\circ\tau))$ and $\diamond = \frac{1}{2}(\mu + \mu\circ(\alpha^{-1}\beta\otimes\alpha\beta^{-1})\circ\tau)$, by expansion in terms of $\mu$
$$\begin{array}{llllll}
&&4 \left(\{\alpha\beta(x),\alpha\beta(y) \diamond \alpha^{2}(z)\} - \{\beta^{2}(x),\alpha\beta(y)\} \diamond \alpha^{2}\beta(z) - \alpha\beta^{2}(y) \diamond \{\alpha\beta(x),\alpha^{2}(z)\}\right)\\
&=&\mu(\alpha\beta^{2}(x)),\mu(\alpha\beta(y),\alpha^{2}(z)))
-\mu(\mu(\beta^{2}(y),\alpha\beta(z)),\beta\alpha^{2}(x)))
+\mu(\alpha\beta^{2}(x)),\mu(\alpha\beta(z),\alpha^{2}(y))) \\
&-&\mu(\mu(\beta^{2}(z),\alpha\beta(y)),\beta\alpha^{2}(x)))
-\mu(\mu(\beta^{2}(x),\alpha\beta(y)),\beta\alpha^{2}(z)))
+\mu(\alpha\beta^{2}(z)),\mu(\alpha\beta(x),\alpha^{2}(y)))\\
&+&\mu(\mu(\beta^{2}(y),\alpha\beta(x)),\beta\alpha^{2}(z)))
-\mu(\alpha\beta^{2}(z)),\mu(\alpha\beta(y),\alpha^{2}(x)))
-\mu(\alpha\beta^{2}(y)),\mu(\alpha\beta(x),\alpha^{2}(z)))\\
&+&\mu(\mu(\beta^{2}(x),\alpha\beta(z)),\beta\alpha^{2}(y)))
+\mu(\alpha\beta^{2}(y)),\mu(\alpha\beta(z),\alpha^{2}(x)))
-\mu(\mu(\beta^{2}(z),\alpha\beta(x)),\beta\alpha^{2}(y)))\\
&=& -as_A(\beta^{2}(x),\alpha\beta(y),\alpha^{2}(z)) - as_A(\beta^{2}(z),\alpha\beta(y),\alpha^{2}(x)) - as_A(\beta^{2}(x),\alpha\beta(z),\alpha^{2}(y))\\
&-& as_A(\beta^{2}(y),\alpha\beta(z),\alpha^{2}(x))  + as_A(\beta^{2}(y),\alpha\beta(x),\alpha^{2}(z)) + as_A(\beta^{2}(z),\alpha\beta(x),\alpha^{2}(y))\\
&=&0~(By~ Lemma \eqref{flex_lemma}).
\end{array}$$
Conversely, assume we have the condition \eqref{flex_Leibniz}. By setting $x = z$, one
gets\\ $as_{A}(\beta^{2}(x), \alpha\beta(y), \alpha^{2}(x)) = 0$. Therefore $A$ is BiHom-flexible.
\end{proof}
Let's give the notion of an admissible BiHom-Poisson algebras.
\begin{defn}\cite{hadimi}
\label{def:admissible}
Let $(A,\mu,\alpha,\beta)$ be a BiHom-algebra.  Then $A$ is called an admissible BiHom-Poisson algebra if it satisfies
\begin{eqnarray}\label{admissibility}
&&as_A(\beta(x),\alpha(y),\alpha^{2}(z)) = \frac{1}{3}\Big\lbrace\mu(\mu(\beta(x),\alpha\beta(z)),\alpha^{2}(y))) - \mu(\mu(\beta^{2}(z),\alpha(x)),\alpha^{2}(y))\nonumber\\
&&+ \mu(\mu(\beta(y),\alpha\beta(z)),\alpha^{2}(x)) - \mu(\mu(\beta(y),\alpha(x)),\beta\alpha^{2}(z))\Big\rbrace
\end{eqnarray}
\end{defn}
It is observed that if $\alpha$ and $\beta$ are inversible, then
(\ref{admissibility}) is equivalent to
 \begin{eqnarray}
 &&as_A(x,y,z) = \frac{1}{3}\Big\lbrace\mu(\mu(x,\alpha^{-1}\beta(z)),\alpha(y))) - \mu(\mu(\alpha^{-2}\beta^{2}(z),\alpha\beta^{-1}(x)),\alpha(y))\nonumber\\
&&+ \mu(\mu(\alpha^{-1}\beta(y),\alpha^{-1}\beta(z)),\alpha^2\beta^{-1}(x)) - \mu(\mu(\alpha^{-1}\beta(y),\alpha^{-1}\beta(x)),\beta(z))\Big\rbrace
 \end{eqnarray}
\begin{prop}
\label{lem:flexible}
Every admissible BiHom-Poisson algebra $(A,\mu,\alpha,\beta)$ is BiHom-flexible, i.e.,
\begin{equation}
\label{homflexible}
as_A(\beta^{2}(x),\alpha\beta(y),\alpha^{2}(z)) + as_A(\beta^{2}(z),\alpha\beta(y),\alpha^{2}(x)) = 0
\end{equation}
for all $x,y,z \in A$.
\end{prop}
\begin{proof}
The required identity \eqref{homflexible} follows immediately from the defining identity \eqref{admissibility}, in which the right-hand side is anti-symmetric in $x$ and $z$.

$\begin{array}{llllll}&&as_A(\beta^{2}(x),\alpha\beta(y),\alpha^{2}(z)) + as_A(\beta^{2}(z),\alpha\beta(y),\alpha^{2}(x))\\&=& \frac{1}{3}\Big\lbrace\mu(\mu(\beta^{2}(x),\alpha\beta(z)),\beta\alpha^{2}(y))) - \mu(\mu(\beta^{2}(z),\alpha\beta(x)),\beta\alpha^{2}(y)))\\&&+ \mu(\mu(\beta^{2}(y),\alpha\beta(z)),\beta\alpha^{2}(x))) -\mu(\mu(\beta^{2}(y),\alpha\beta(x)),\beta\alpha^{2}(z)))\\&&+\mu(\mu(\beta^{2}(z),\alpha\beta(x)),\beta\alpha^{2}(y))) - \mu(\mu(\beta^{2}(x),\alpha\beta(z)),\beta\alpha^{2}(y)))\\&&+ \mu(\mu(\beta^{2}(y),\alpha\beta(x)),\beta\alpha^{2}(z))) - \mu(\mu(\beta^{2}(y),\alpha\beta(z)),\beta\alpha^{2}(x)))\Big\rbrace=0.\end{array}$
\end{proof}

Next we observe that in an admissible BiHom-Poisson algebra the cyclic sum of the BiHom-associator is trivial.
\begin{prop}
\label{lem:cyclichomass}
Let $(A,\mu,\alpha,\beta)$ be an admissible BiHom-Poisson algebra.  Then
\begin{equation}
\label{S}
S_A(x,y,z): = as_{A}(\beta^{2}(x),\alpha\beta(y),\alpha^{2}(z))  + as_{A}(\beta^{2}(y),\alpha\beta(z),\alpha^{2}(x))  + as_{A}(\beta^{2}(z),\alpha\beta(x),\alpha^{2}(y))  = 0
\end{equation}
for all $x,y,z \in A$.
\end{prop}
\begin{proof}
Using the defining identity \eqref{admissibility}, we have:
\[
\begin{array}{llllll}
as_A(\beta^{2}(x),\alpha\beta(y),\alpha^{2}(z))
&=& \frac{1}{3}\Big(\mu(\mu(\beta^{2}(x),\alpha\beta(z)),\beta\alpha^{2}(y)))- \mu(\mu(\beta^{2}(z),\alpha\beta(x)),\beta\alpha^{2}(y)))\\&&+ \mu(\mu(\beta^{2}(y),\alpha\beta(z)),\beta\alpha^{2}(x))) - \mu(\mu(\beta^{2}(y),\alpha\beta(x)),\beta\alpha^{2}(z)))\Big)\\
&=& -\frac{1}{3}\Big(\mu(\mu(\beta^{2}(z),\alpha\beta(y)),\beta\alpha^{2}(x))) - \mu(\mu(\beta^{2}(y),\alpha\beta(z)),\beta\alpha^{2}(x)))\\&&+ \mu(\mu(\beta^{2}(x),\alpha\beta(y)),\beta\alpha^{2}(z))) - \mu(\mu(\beta^{2}(x),\alpha\beta(z)),\beta\alpha^{2}(y)))\Big)\\
&&+\frac{1}{3}\Big(\mu(\mu(\beta^{2}(x),\alpha\beta(y)),\beta\alpha^{2}(z))) - \mu(\mu(\beta^{2}(y),\alpha\beta(x)),\beta\alpha^{2}(z)))\\&&+ \mu(\mu(\beta^{2}(z),\alpha\beta(y)),\beta\alpha^{2}(x))) - \mu(\mu(\beta^{2}(z),\alpha\beta(x)),\beta\alpha^{2}(y)))\Big)\\
&=& - as_A(\beta^{2}(z),\alpha\beta(x),\alpha^{2}(y)) + as_A(\beta^{2}(x),\alpha\beta(z),\alpha^{2}(y))\\
&= &- as_A(\beta^{2}(z),\alpha\beta(x),\alpha^{2}(y)) - as_A(\beta^{2}(y),\alpha\beta(z),\alpha^{2}(x))~(by~(\ref{homflexible})).
\end{array}
\]
Therefore, we conclude that $S_A = 0$.
\end{proof}

 \section{Derivations  of BiHom-Poisson algebras}
\def\theequation{\arabic{section}. \arabic{equation}}
\setcounter{equation} {0}

In this section, we introduce and study derivations, generalized derivations and quasiderivations of  BiHom-Poisson algebras.

\begin{defn}Let $(A,\{\cdot,\cdot\},\mu, \a,\b)$  be a BiHom-Poisson algebra.
 A  linear map $D: A\rightarrow A$  is called an  $(\a^k,\b^{l})$-derivation of $A$  if it satisfies
\begin{enumerate}
\item $D \circ \alpha = \alpha \circ D,\; D \circ \beta = \beta \circ D; $
\item $D(\{x, y\})=\{\a^{k}\b^{l}(x), D(y)\}+\{D(x), \a^{k}\b^{l}(y)\}$;
\item $D(\mu(x,y))=\mu(\a^{k}\b^{l}(x), D(y))+\mu(D(x),\a^{k}\b^{l}(y))$,
\end{enumerate}
for all $x,y\in A$.
\end{defn}

We denote by $Der(A):=\displaystyle{\bigoplus_{k\geq 0}}\displaystyle{\bigoplus_{l\geq 0}}Der_{(\a^k,\b^{l})}(A)$, where $Der_{(\a^k,\b^{l})}(A)$ is the set of all $(\a^k,\b^{l})$-derivations of $A$.
Obviously, $Der(A)$ is a subalgebra of $End(A)$.
\begin{lem}\label{le:2.7}
Let $(A,\{\cdot,\cdot\}, \mu, \alpha,\beta)$ be a BiHom-Poisson algebra. We define a subspace $\mathcal{W}$ of $End(A)$ by $\mathcal{W} = \{w \in End(A) |~ w \circ \alpha = \alpha \circ w~and~w \circ \beta = \beta \circ w\}$ and $\sigma_{1},\sigma_{2} : \mathcal{W} \rightarrow \mathcal{W}$  linear maps satisfying $\sigma_{1}(w) = \alpha \circ w$ and $\sigma_{2}(w) = \beta \circ w$. Then
a quadruple $(\mathcal{W}, [\cdot,\cdot], \sigma_{1},\sigma_{2})$, where the multiplication $[\cdot,\cdot] : \mathcal{W} \times \mathcal{W} \rightarrow \mathcal{W}$ is defined for
$w_{1}, w_{2} \in \mathcal{W}$ by
 \[   [w_{1}, w_{2}]= w_{1} \circ w_{2} - w_{2} \circ w_{1},\]
    is a BiHom-Lie algebra.
\end{lem}
\begin{proof}
For any $w_{1},w_{2},w_{3}\in\mathcal{W},~~k_{1},k_{2}\in\mathbb{K},$ we have

$ [w_{1}, w_{1}]= w_{1} \circ w_{1} - w_{1} \circ w_{1}=0,$\\

$\begin{array}{lll}
&&[k_{1}w_{1}+k_{2}w_{2},w_{3}]
=(k_{1}w_{1}+k_{2}w_{2})w_{3}-w_{3}(k_{1}w_{1}+k_{2}w_{2})\\
&=&k_{1}(w_{1}w_{3}-w_{3}w_{1})+k_{2}(w_{2}w_{3}-w_{3}w_{2})
=k_{1}[w_{1},w_{3}]+k_{2}[w_{2},w_{3}],
\end{array}$\\\\

$\begin{array}{lll}
&&[\sigma_{2}(w_{1}),\sigma_{1}(w_{2})]
=[\beta(w_{1}),\alpha(w_{2})]
=\alpha\beta (w_{1} w_{2}-w_{2}w_{1})
=-\alpha\beta (w_{2} w_{1}-w_{1}w_{2})\\
&=&-[\sigma_{2}(w_{2}),\sigma_{1}(w_{1})]
\end{array}$\\\\

$\begin{array}{llllllll}
&&[\sigma_{2}^{2}(w_{1}),[\sigma_{2}(w_{2}),\sigma_{1}(w_{3})]]+[\sigma_{2}^{2}(w_{3}),[\sigma_{2}(w_{1}),\sigma_{1}(w_{2})]]\\
&&+[\sigma_{2}^{2}(w_{2}),[\sigma_{2}(w_{3}),\sigma_{1}(w_{1})]]\\
&=&\beta^{3}\alpha w_{1}w_{2}w_{3}-\beta^{3}\alpha w_{1}w_{3}w_{2}-\beta^{3}\alpha w_{2}w_{3}w_{1}+\beta^{3}\alpha w_{3}w_{2}w_{1}\\&&+
\beta^{3}\alpha w_{2}w_{3}w_{1}-\beta^{3}\alpha w_{2}w_{1}w_{3}-\beta^{3}\alpha w_{3}w_{1}w_{2}+\beta^{3}\alpha w_{1}w_{3}w_{2}\\&&+
\beta^{3}\alpha w_{3}w_{1}w_{2}-\beta^{3}\alpha w_{3}w_{2}w_{1}-\beta^{3}\alpha w_{1}w_{2}w_{3}+\beta^{3}\alpha w_{2}w_{1}w_{3}\\
&=&0.
\end{array}$

Then $(\mathcal{W}, [\cdot,\cdot], \sigma_{1},\sigma_{2})$  is a Bihom-Lie algebra.
\end{proof}

\begin{thm}\label{le:2.9}Let $(A,\{\cdot,\cdot\},\mu, \a,\b)$  be a BiHom-Poisson algebra.
For any $D \in Der_{(\alpha^{k},\beta^{l})}(A)$ and $D^{'} \in Der_{(\alpha^{k'},\beta^{l'})}(A)$, define their commutator $[D, D^{'}]$ as usual:
\[[D, D^{'}] = D \circ D^{'} - D^{'} \circ D.\]
Then $[D, D^{'}] \in Der_{(\alpha^{k + k'},\beta^{l+l'})}(A)$.
\end{thm}
\begin{proof}
It is sufficient to prove $[Der_{(\a^k,\b^l)}(A), Der_{(\a^{k'},\b^{l'})}(A)]\subseteq Der_{(\a^{k+k'},\b^{l+l')}}(A)$. It is easy to check that $[D, D']\circ \a=\a\circ [D, D']$ and $[D, D']\circ \b=\b\circ [D, D']$.

For any $x,y\in A$, we have
$$\begin{array}{lllll}
&&[D,D'](\{x,y\})\\
&=&D\circ D'(\{x,y\})-D'\circ D(\{x,y\})\\
&=&D(\{D'(x),\alpha^{k'}\beta^{l'}(y)\}+\{\alpha^{k'}\beta^{l'}(x),D'(y)\})\\
&&-D'(\{D'(x), \alpha^{k}\beta^{l}(y)\}+\{\alpha^{k}\beta^{l}(x),D(y)\})\\
&=&D(\{D'(x),\alpha^{k'}\beta^{l'}(y)\})+D((\{\alpha^{k'}\beta^{l'}(x),D'(y)\})\\
&&-D'(\{D(x)(x),\alpha^{k}\beta^{l}(y)\})-D'(\{\alpha^{k}\beta^{l}(x),D(y)\})\\
&=&\{D\circ D'(x),\alpha^{k+k'}\beta^{l+l'}(y)\}+\{\alpha^{k}\beta^{l}\circ D'(x),D\circ\alpha^{k'}\beta^{l'}(y)\}\\
&&+\{D\circ\alpha^{k'}\beta^{l'}(x),\alpha^{k}\beta^{l}\circ D'(y)\}+\{\alpha^{k+k'}\beta^{l+l'}(x),D\circ D'(y)\}\\
&&-\{D'\circ D(x),\alpha^{k+k'}\beta^{l+l'}(y)\}-\{\alpha^{k'}\beta^{l'}\circ D(x),D'\circ \alpha^{k}\beta^{l}(y)\}\\
&&-\{D'\circ\alpha^{k}\beta^{l}(x),\alpha^{k'}\beta^{l'}\circ D(y)\}-\{\alpha^{k+k'}\beta^{l+l'}(x),D'\circ D(y)\}.
\end{array}$$

Similarly, we can check that
\begin{eqnarray*}
[D, D'](\mu(x, y))
\mu([D, D'](x), \a^{k+k'}\b^{l+l'}(y))+\mu(\a^{k+k'}\b^{l+l'}(x), [D,D'](y)).
\end{eqnarray*}
It follows that $[D, D']\in Der_{(\a^{k+k'},\b^{l+l'})}(A)$.\end{proof}
\begin{defn}
Let $(A,\{\cdot,\cdot\},\mu, \a,\b)$  be a BiHom-Poisson algebra.
$D\in End(A)$ is said to be a generalized $(\a^k,\b^{l})$-derivation of $A$, if there  exists two endomorphisms
$D', D''\in End(A)$ such that
\begin{enumerate}
\item $D \circ \alpha = \alpha \circ D,\; D \circ \beta = \beta \circ D\; $;
\item $D^{'} \circ \alpha = \alpha \circ D^{'},\; D' \circ \beta = \beta \circ D^{'}\;$;
\item $D^{''} \circ \alpha = \alpha \circ D^{''},\;D^{''} \circ \beta = \beta \circ D^{''}$;
\item $\{D(x), \a^{k}\b^l (y)\}+\{\a^{k}\b^ l (x), D'(y)\}=D''(\{x, y\})$;
\item $\mu(D(x), \alpha^{k}\beta^{l}(y)) + \mu(\alpha^{k}\beta^{l}(x), D^{'}(y)) = D''(\mu(x, y))$,
\end{enumerate}
for all $x,y\in A$.

The set of generalized $(  \alpha^{k},\beta^{l}) $-derivations of $ A $ is $ {\rm GDer}_{(  \alpha^{k},\beta^{l})}(A) $ and we denote  \[{\rm GDer}(A):=\bigoplus_{k \geq 0}\bigoplus_{l \geq 0} {\rm GDer}_{(\alpha^{k},\beta^{l})}(A).\]
\end{defn}
\begin{defn}
Let $(A,\{\cdot,\cdot\},\mu, \a,\b)$  be a BiHom-Poisson algebra.
$D\in End(A)$ is said to be an $(\a^k,\b^{l})$-quasiderivation of $A$, if there  exists endomorphisms
$D', D''\in End(A)$ such that
\begin{enumerate}
\item $D \circ \alpha = \alpha \circ D,\; D \circ \beta = \beta \circ D\; $;
\item $D^{'} \circ \alpha = \alpha \circ D^{'},\; D^{'} \circ \beta = \beta \circ D^{'}\;$;
\item $D^{''} \circ \alpha = \alpha \circ D^{''},\;D^{''} \circ \beta = \beta \circ D^{''}$;
\item $\{D(x), \a^{k}\b^l (y)\}+\{\a^{k}\b^ l (x), D(y)\}=D{'}(\{x, y\}),$
\item $\mu(D(x), \alpha^{k}\beta^{l}(y)) + \mu(\alpha^{k}\beta^{l}(x), D(y)) = D^{''}(\mu(x, y)),$
\end{enumerate}
for all $x,y\in A$.
\end{defn}

We then define
\[{\rm QDer}(A):=\bigoplus_{k \geq 0}\bigoplus_{l \geq 0} {\rm QDer}_{(  \alpha^{k},\beta^{l})}(A).\]

\begin{defn}Let $(A,\{\cdot,\cdot\},\mu, \a,\b)$  be a BiHom-Poisson algebra.
 A  linear map $D: A\rightarrow A$  is called an  $(\alpha^{k},\beta^{l})$-centroid of $A$  if it satisfies
\begin{enumerate}
\item $D \circ \alpha = \alpha \circ D,~D \circ \beta = \beta \circ D$;
\item $\{D(x), \alpha^{k}\beta^{l}(y)\} = \{\alpha^{k}\beta^{l}(x), D(y)\} = D(\{x, y\})$;
\item $\mu(D(x), \alpha^{k}\beta^{l}(y)) = \mu(\alpha^{k}\beta^{l}(x), D(y)) = D(\mu(x, y)),\quad\forall x, y \in A$.
\end{enumerate}
\end{defn}
We set
$$C(A):=\bigoplus_{k\geq 0}\bigoplus_{l\geq 0} C_{(\alpha^{k},\beta^{l})}(A).$$
\begin{defn}
The $(\alpha^{k},\beta^{l}) $-quasicentroid of a Bihom-Poisson algebra $(A,\{\cdot,\cdot\},\mu,\alpha,\beta)$ denoted by $ {\rm QC}_{(  \alpha^{k},\beta^{l})} (A)$ is the set of linear maps $ D $ such that
\begin{enumerate}
\item $D \circ \alpha = \alpha \circ D,~D \circ \beta = \beta \circ D$;
\item $\{D(x), \alpha^{k}\beta^{l}(y)\} = \{\alpha^{k}\beta^{l}(x), D(y)\}$;
\item $\mu(D(x), \alpha^{k}\beta^{l}(y)) = \mu(\alpha^{k}\beta^{l}(x), D(y)),\quad\forall x, y \in A$.
\end{enumerate}
\end{defn}
We set
\[{\rm QC}(A):=\bigoplus_{k \geq 0}\bigoplus_{l \geq 0} {\rm QC}_{(\alpha^{k},\beta^{l})}(A).\]

\begin{rmk}
Let $(A,\{\cdot,\cdot\},\mu,\alpha,\beta)$  be a BiHom-Poisson algebra.
Then $C(A)\subseteq QC(A).$
\end{rmk}

\begin{defn}
A linear map $ D $ is called an $(  \alpha^{k},\beta^{l}) $-central derivation of $ A $ if it satisfies
\begin{enumerate}
\item $D \circ \alpha = \alpha \circ D,~D \circ \beta = \beta\circ D$;
\item $\{D(x), \alpha^{k}\beta^{l}(y)\} = D(\{x, y\}) = 0;$
\item $\mu(D(x), \alpha^{k}\beta^{l}(y)) = D(\mu(x, y)) = 0,\quad\forall x, y \in A$.
\end{enumerate}
\end{defn}
The set of $(\alpha^{k},\beta^{l}) $-central derivations is denoted by $ {\rm ZDer}_{(  \alpha^{k},\beta^{l})}(A) $ and we set
\[{\rm ZDer}(A):=\bigoplus_{k \geq 0}\bigoplus_{l \geq 0} {\rm ZDer}_{(  \alpha^{k},\beta^{l})}(A).\]

\begin{rmk}
Let $(A,\{\cdot,\cdot\},\mu,\alpha,\beta)$  be a BiHom-Poisson algebra.
Then
\begin{eqnarray*}
ZDer(A)\subseteq Der(A)\subseteq QDer(A)\subseteq GDer(A) \subseteq End(A).
\end{eqnarray*}
\end{rmk}

\begin{prop}
Let $(A,\{\cdot,\cdot\},\mu,\a,\b)$  be a BiHom-Poisson algebra.
 Then the following statements hold:
\begin{enumerate}
\item $GDer(A)$, $QDer(A)$ and $C(A)$ are BiHom-subalgebras of $(\mathcal{W}, [\cdot,\cdot], \sigma_{1},\sigma_{2})$;
\item $ZDer(A)$ is a BiHom-ideal of $Der(A)$.
\end{enumerate}
\end{prop}
\begin{proof}
 We only prove that $GDer(A)$ is a subalgebra of $\mathcal{W}.$ The case of $ QDer(L)$ and $C(L)$ is similar.
\begin{enumerate}
\item
Suppose that $D_{1} \in GDer_{(\alpha^{k},\beta^{l})}(A)$, $D_{2} \in GDer_{(\alpha^{k'},\beta^{l'})}(A)$. Then for any $x, y \in A$,
\begin{align*}
&\{\sigma_{1}(D_{1})(x), \alpha^{k + 1}\beta^{l}(y)\} = \{\alpha_{1} \circ D_{1}(x), \alpha^{k + 1}\beta^{l}(y)\} = \alpha(\{D_{1}(x), \alpha^{k}\beta^{l}(y)\})\\
&= \alpha(D_{1}^{''}(\{x, y\}) - \{\alpha^{k}\beta^{l}(x), D_{1}^{'}(y)\}) = \sigma_{1}(D_{1}^{''})(\{x, y\}) - \{\alpha^{k + 1}\beta^{l}(x), \sigma_{1}(D_{1}^{'})(y)\}.
\end{align*}
Since $\sigma_{1}(D_{1}^{''}), \sigma_{1}(D_{1}^{'}) \in End(A)$, we have $\sigma_{1}(D_{1}) \in GDer_{(\alpha^{k + 1},\beta^{l})}(A)$,
\begin{align*}
&\{\sigma_{2}(D_{1})(x), \alpha^{k}\beta^{l+1}(y)\} = \{\beta \circ D_{1}(x), \alpha^{k}\beta^{l+1}(y)\} = \beta(\{D_{1}(x), \alpha^{k}\beta^{l}(y)\})\\
&= \beta(D_{1}^{''}(\{x, y\}) - \{\alpha^{k}\beta^{l}(x), D_{1}^{'}(y)\}) = \sigma_{2}(D_{1}^{''})(\{x, y\}) - \{\alpha^{k}\beta^{l+1}(x), \sigma_{2}(D_{1}^{'})(y)\}.
\end{align*}
Since $\sigma_{2}(D_{1}^{''}), \sigma_{2}(D_{1}^{'}) \in End(A)$, we have $\sigma_{2}(D_{1}) \in GDer_{(\alpha^{k},\beta^{l+1})}(A)$.
\begin{align*}
&\{\{D_{1}, D_{2}\}(x), \alpha^{k+k'}\beta^{l+l'}(y)\}\\
&= \{D_{1} \circ D_{2}(x), \alpha^{k+k'}\beta^{l+l'}(y)\} - \{D_{2} \circ D_{1}(x), \alpha^{k+k'}\beta^{l+l'}(y)\}\\
&= D_{1}^{''}(\{D_{2}(x), \alpha^{k'}\beta^{l'}(y)\}) - \{\alpha^{k}\beta^{l}(D_{2}(x)), D_{1}^{'}(\alpha^{k'}\beta^{l'}(y))\} - D_{2}^{''}(\{D_{1}(x), \alpha^{k}\beta^{l}(y)\})\\ &+ \{\alpha^{k'}\beta^{l'}(D_{1}(x)), D_{2}^{'}(\alpha^{k}\beta^{l}(y))\}\\
&= D_{1}^{''}(D_{2}^{''}(\{x, y\}) - \{\alpha^{k'}\beta^{l'}(x), D_{2}^{'}(y)\}) -\{\alpha^{k'}\beta^{l'}(D_{1}(x)), D_{2}^{'}(\alpha^{k}\beta^{l}(y))\}\\
&= D_{1}^{''} \circ D_{2}^{''}(\{x, y\}) - D_{1}^{''}(\{\alpha^{k'}\beta^{l'}(x), D_{2}^{'}(y)\}) - \{\alpha^{k}\beta^{l}(D_{2}(x)), D_{1}^{'}(\alpha^{k'}\beta^{l'}(y))\}\\
&- D_{2}^{''} \circ D_{1}^{''}(\{x, y\}) + D_{2}^{''}(\{\alpha^{k}\beta^{l}(x), D_{1}^{'}(y)\}) + \{\alpha^{k'}\beta^{l'}(D_{1}(x)), D_{2}^{'}(\alpha^{k}\beta^{l}(y))\}\\
&= D_{1}^{''} \circ D_{2}^{''}(\{x, y\}) - \{D_{1}(\alpha^{k'}\beta^{l'}(x)), \alpha^{k}\beta^{l}(D_{2}^{'}(y))\} - \{\alpha^{k +k'}\beta^{l+l'}(x), D_{1}^{'}(D_{2}^{'}(y))\}\\
&- \{\alpha^{k}\beta^{l}(D_{2}(x)), D_{1}^{'}(\alpha^{k'}\beta^{l'}(y))\} - D_{2}^{''} \circ D_{1}^{''}(\{x, y\}) + \{D_{2}(\alpha^{k}\beta^{l}(x)), \alpha^{k'}\beta^{l'}(D_{1}^{'}(y))\}\\
&+ \{\alpha^{k+k'}\beta^{l+l'}(x), D_{2}^{'}(D_{1}^{'}(y))\} + \{\alpha^{k'}\beta^{l'}(D_{1}(x)), D_{2}^{'}(\alpha^{k}\beta^{l}(y))\}\\
&= D_{1}^{''} \circ D_{2}^{''}(\{x, y\}) - D_{2}^{''} \circ D_{1}^{''}(\{x, y\}) - \{\alpha^{k + k'}\beta^{l+l'}(x), D_{1}^{'}(D_{2}^{'}(y))\}\\
&+ \{\alpha^{k + k'}\beta^{l+l'}(x), D_{2}^{'}(D_{1}^{'}(y))\}\\
&= \{D_{1}^{''}, D_{2}^{''}\}(\{x, y\}) - \{\alpha^{k + k'}\beta^{l+l'}(x), \{D_{1}^{'}, D_{2}^{'}\}(y)\}.
\end{align*}
Since $\{D_{1}^{''}, D_{2}^{''}\}, \{D_{1}^{'}, D_{2}^{'}\} \in End(A)$, we have $\{D_{1}, D_{2}\} \in GDer_{(\alpha^{k + k'},\beta^{l+l'})}(A)$.

Similarly, we have $\mu(D_{1}, D_{2}) \in GDer_{(\alpha^{k + k'},\beta^{l+l'})}(A)$.

Therefore, $GDer(A)$ is a BiHom-subalgebra of $(\mathcal{W}, [\cdot,\cdot], \sigma_{1},\sigma_{2})$.
\item
Suppose that $D_{1} \in ZDer_{(\alpha^{k},\beta^{l})}(A)$, $D_{2} \in Der_{(\alpha^{k'},\beta^{l'})}(A)$. Then for any $x, y \in A$,
\[\sigma_{1}(D_{1})(\mu(x, y)) = \alpha \circ D_{1}(\mu(x, y)) = 0.\]
\[\sigma_{1}(D_{1})(\mu(x, y)) = \alpha \circ D_{1}(\mu(x, y)) = \alpha(\mu(D_{1}(x), \alpha^{k}\beta^{l}(y))) = \mu(\sigma_{1}(D_{1})(x), \alpha^{k + 1}\beta^{l}(y)).\]
Similarly, $\sigma_{1}(D_{1})(\{x, y\})= \{\sigma_{1}(D_{1})(x), \alpha^{k + 1}\beta^{l}(y)\}.$

Hence, $\sigma_{1}(D_{1}) \in ZDer_{(\alpha^{k+1 },\beta^{l})}(A)$.

In the same way, we have, $\sigma_{2}(D_{1}) \in ZDer_{(\alpha^{k},\beta^{l+1})}(A)$. Next, we have
\begin{align*}
&[D_{1}, D_{2}](\mu(x, y))\\
&= D_{1} \circ D_{2}(\mu(x, y)) - D_{2} \circ D_{1}(\mu(x, y))\\
&= D_{1}(\mu(D_{2}(x), \alpha^{k'}\beta^{l'}(y)) + \mu(\alpha^{k'}\beta^{l'}(x), D_{2}(y))) = 0.
\end{align*}
\begin{align*}
&[D_{1}, D_{2}](\mu(x, y))\\
&= D_{1} \circ D_{2}(\mu(x, y)) - D_{2} \circ D_{1}(\mu(x, y))\\
&= D_{1}(\mu(D_{2}(x), \alpha^{k'}\beta^{l'}(y)) + \mu(\alpha^{k'}\beta^{l'}(x), D_{2}(y))) - D_{2}(\mu(D_{1}(x), \alpha^{k}\beta^{l}(y)))\\
&= \mu(D_{1}(D_{2}(x)), \alpha^{k+k'}\beta^{l+l'}(y)) + \mu(D_{1}(\alpha^{k'}\beta^{l'}(x)), \alpha^{k}\beta^{l}(D_{2}(y))) - \mu(D_{2}(D_{1}(x)), \alpha^{k +k'}\beta^{l+l'}(y))\\
&- \mu(\alpha^{k'}\beta^{l'}(D_{1}(x)), D_{2}(\alpha^{k}\beta^{l}(y)))\\
&= \mu([D_{1}, D_{2}](x), \alpha^{k + k'}\beta^{l+l'}(y)).
\end{align*}
Similarly, $[D_{1}, D_{2}](\{x, y\})=\{[D_{1}, D_{2}](x), \alpha^{k + k'}\beta^{l+l'}(y)\}=0$.

Hence, we have $[D_{1}, D_{2}] \in ZDer_{(\alpha^{k + k'},\beta^{l+l'})}(A)$.
Therefore, $ZDer(A)$ is a BiHom-ideal of $Der(A)$.
\end{enumerate}
\end{proof}

\begin{lem}
Let $(A,\{\cdot,\cdot\},\mu, \a,\b)$  be a BiHom-Poisson algebra,
 then the following statements hold:
\begin{enumerate}
\item $[Der(A), C(A)]\subseteq C(A)$.

\item $[QDer(A), QC(A)]\subseteq QC(A)$.

\item
$[QC(A), QC(A)]\subseteq QDer(A)$.

\item
$C(A)\subseteq QDer(A)$.

\item
$QDer(A)+QC(A)\subseteq GDer(A)$.

\item
$C(A) \circ Der(A) \subseteq Der(A)$.
\end{enumerate}
\end{lem}
\begin{proof}
\begin{enumerate}
\item
Suppose that $D_{1} \in Der_{(\alpha^{k},\beta^{l})}(A)$, $D_{2} \in C_{(\alpha^{k'}\beta^{l'})}(A)$. Then for any $x, y \in A$, we have
\begin{align*}
&[D_{1}, D_{2}](\{x, y\})\\
&= D_{1} \circ D_{2}(\{x, y\}) - D_{2} \circ D_{1}(\{x, y\})\\
&= D_{1}(\{D_{2}(x), \alpha^{k'}\beta^{l'}(y)\}) - D_{2}(\{D_{1}(x), \alpha^{k}\beta^{l}(y)\} + \{\alpha^{k}\beta^{l}(x), D_{1}(y)\})\\
&= \{D_{1}(D_{2}(x)), \alpha^{k + k'}\beta^{l+l'}(y)\} + \{\alpha^{k}\beta^{l}(D_{2}(x)), D_{1}(\alpha^{k'}\beta^{l'}(y))\} - \{D_{2}(D_{1}(x)), \alpha^{k + k'}\beta^{l+l'}(y)\}\\
&- \{D_{2}(\alpha^{k}\beta^{l}(x)), \alpha^{k'}\beta^{l'}(D_{1}(y))\}\\
&= \{[D_{1}, D_{2}](x), \alpha^{k + k'}\beta^{l+l'}(y)\}.
\end{align*}
Similarly, we have $[D_{1}, D_{2}](\{x, y\}) = \{\alpha^{k + k'}\beta^{l+l'}(x), [D_{1}, D_{2}](y)\}$.

In the same way, we have
$[D_{1}, D_{2}](\mu(x, y))= \mu([D_{1}, D_{2}](x), \alpha^{k + k'}\beta^{l+l'}(y))\\ = \mu(\alpha^{k + k'}\beta^{l+l'}(x), [D_{1}, D_{2}](y))$.

Hence, $[D_{1}, D_{2}] \in C_{(\alpha^{k + k'},\beta^{l+l'})}(A)$. Therefore $[D_{1}, D_{2}] \in C(A)$.

\item
 Suppose that $D_{1} \in QDer_{(\alpha^{k},\beta^{l})}(A)$, $D_{2} \in QC_{(\alpha^{k'},\beta^{l'})}(A)$. Then for any $x, y \in A$, we have
\begin{align*}
&\{[D_{1}, D_{2}](x), \alpha^{k + k'}\beta^{l+l'}(y)\}\\
&= \{D_{1} \circ D_{2}(x), \alpha^{k + k'}\beta^{l+l'}(y)\} - \{D_{2} \circ D_{1}(x), \alpha^{k + k'}\beta^{l+l'}(y)\}\\
&= D_{1}^{'}(\{D_{2}(x), \alpha^{k'}\beta^{l'}(y)\}) - \{\alpha^{k}\beta^{l}(D_{2}(x)), D_{1}(\alpha^{k'}\beta^{l'}(y))\} \\&- \{\alpha^{k'}\beta^{l'}(D_{1}(x)), D_{2}(\alpha^{k}\beta^{l}(y))\}\\
&= D_{1}^{'}(\{\alpha^{k'}\beta^{l'}(x), D_{2}(y)\}) - \{\alpha^{k}\beta^{l}(D_{2}(x)), D_{1}(\alpha^{k'}\beta^{l'}(y))\}\\& - \{\alpha^{k'}\beta^{l'}(D_{1}(x)), D_{2}(\alpha^{k}\beta^{l}(y))\}\\
&= \{D_{1}(\alpha^{k'}\beta^{l'}(x)), \alpha^{k}\beta^{l}(D_{2}(y))\} + \{\alpha^{k + k'}\beta^{l+l'}(x), D_{1}(D_{2}(y))\} \\&- \{\alpha^{k}\beta^{l}(D_{2}(x)), D_{1}(\alpha^{k'}\beta^{l'}(y))\}
- \{\alpha^{k'}\beta^{l'}(D_{1}(x)), D_{2}(\alpha^{k}\beta^{l}(y))\}\\
&= \{\alpha^{k + k'}\beta^{l+l'}(x), D_{1}(D_{2}(y))\} - \{\alpha^{k}\beta^{l}(D_{2}(x)), D_{1}(\alpha^{k'}\beta^{l'}(y))\}\\
&= \{\alpha^{k + k'}\beta^{l+l'}(x), D_{1}(D_{2}(y))\} - \{D_{2}(\alpha^{k}\beta^{l}(x)), \alpha^{k'}\beta^{l'}(D_{1}(y))\}\\
&= \{\alpha^{k + k'}\beta^{l+l'}(x), D_{1}(D_{2}(y))\} - \{\alpha^{k + k'}\beta^{l+l'}(x), D_{2}(D_{1}(y))\}\\
&= \{\alpha^{k + k'}\beta^{l+l'}(x), [D_{1}, D_{2}](y)\}.
\end{align*}
Similarly, $\mu([D_{1}, D_{2}](x), \alpha^{k + k'}\beta^{l+l'}(y))= \mu(\alpha^{k + k'}\beta^{l+l'}(x), [D_{1}, D_{2}](y)).$

Hence, we have $[D_{1}, D_{2}] \in QC_{(\alpha^{k + k'},\beta^{l+l'})}(A)$. So $[QDer(A), QC(A)] \subseteq QC(A)$.

\item
Suppose that $D_{1} \in QC_{(\alpha^{k},\beta^{l})}(A)$, $D_{2} \in QC_{(\alpha^{k'},\beta^{l'})}(A)$. Then for any $x, y \in A$, we have
\begin{align*}
&\{[D_{1}, D_{2}] (x), \alpha^{k + k'}\beta^{l+l'}(y)\}\\
&= \{D_{1} \circ D_{2}(x), \alpha^{k + k'}\beta^{l+l'}(y)\} - \{D_{2} \circ D_{1}(x), \alpha^{k + k'}\beta^{l+l'}(y)\}\\
&= \{\alpha^{k}\beta^{l}(D_{2}(x)), D_{1}(\alpha^{k'}\beta^{l'}(y))\} - \{\alpha^{k'}\beta^{l'}(D_{1}(x)), D_{2}(\alpha^{k}\beta^{l}(y))\}\\
&= \{D_{2}(\alpha^{k}\beta^{l}(x)), \alpha^{k'}\beta^{l'}(D_{1}(y))\} - \{D_{1}(\alpha^{k'}\beta^{l'}(x)), \alpha^{k}\beta^{l}(D_{2}(y))\}\\
&= \{\alpha^{k + k'}\beta^{l+l'}(x), D_{2}(D_{1}(y))\} - \{\alpha^{k + k'}\beta^{l+l'}(x), D_{1}(D_{2}(y))\}\\
&= -\mu(\alpha^{k + k'}\beta^{l+l'}(x), [D_{1}, D_{2}] (y)),
\end{align*}
i.e., $\{[D_{1}, D_{2}] (x), \alpha^{k + k'}\beta^{l+l'}(y)\} + \{\alpha^{k + k'}\beta^{l+l'}(x), [D_{1}, D_{2}]\} = 0$.

Similarly, $\mu([D_{1}, D_{2}] (x), \alpha^{k + k'}\beta^{l+l'}(y)) + \mu(\alpha^{k + k'}\beta^{l+l'}(x), [D_{1}, D_{2}] ) = 0$.
Hence, we have $[D_{1}, D_{2}]\in QDer_{(\alpha^{k + k'},\beta^{l+l'})}(A)$, which implies that $[QC(A), QC(A)] \subseteq QDer(A)$.
\item
Suppose that $D \in C_{(\alpha^{k},\beta^{l})}(A)$. Then for any $x, y \in A$, we have
\[D(\{x, y\}) = \{D(x), \alpha^{k}\beta^{l}(y)\} = \{\alpha^{k}\beta^{l}(x), D(y)\}.\]

Hence, we have
\[\{D(x), \alpha^{k}\beta^{l}(y)\} + \{\alpha^{k}\beta^{l}(x), D(y)\} = 2D(\{x, y\}),\]
Similarly, $\mu(D(x), \alpha^{k}\beta^{l}(y)) + \mu(\alpha^{k}\beta^{l}(x), D(y)) = 2D(\mu(x, y)),$

which implies that $D \in QDer_{(\alpha^{k},\beta^{l})}(A)$. So $C(A) \subseteq QDer(A)$.
\item
In fact. Let $D_1\in QDer_{(\a^{k},\b^{l})}(A), D_2\in QC_{(\a^{k},\b^l)}(A)$. Then there exist $D'_1, D''_1\in End(A)$, for any $x,y\in A$, we have
\begin{eqnarray*}
&&\{D_1(x), \a^{k}\b^{l}(y)\}+\{\a^{k}\b^{l}(x), D_1(y)\}=D_1'(\{x, y\}),\\
&&\mu(D_1(x),\a^{k}\b^{l}(y))+\mu(\a^{k}\b^{l}(x), D_1(y))
=D_1''(\mu(x,y)).
\end{eqnarray*}
Thus, for any $x, y\in A$, we have
\begin{eqnarray*}
\{(D_1+D_2)(x), \a^{k}\b^{l}(y)\}&=&\{D_1(x), \a^{k}\b^{l}(y)\}+\{ D_2(x), \a^{k}\b^{l}(y)\}\\
&=& D_1'(\{x, y\})-\{\a^{k}\b^{l}(x), D_1(y)\}+\{ \a^{k}\b^{l}(x),D_2(y)\}\\
&=& D_1'(\{x, y\})-\{\a^{k}\b^{l}(x), (D_1-D_2)(y)\},
\end{eqnarray*}
and
\begin{eqnarray*}
\mu((D_1+D_2)(x), \a^{k}\b^{l}(y))&=&\mu(D_1(x), \a^{k}\b^{l}(y))+\mu( D_2(x), \a^{k}\b^{l}(y)\}\\
&=& D_1''(\mu(x, y))-\mu(\a^{k}\b^{l}(x), D_1(y))+\mu( \a^{k}\b^{l}(x),D_2(y))\\
&=& D_1''(\mu(x, y))-\mu(\a^{k}\b^{l}(x), (D_1-D_2)(y)),
\end{eqnarray*}
Therefore, $D_1+D_2\in GDer_{(\a^{k},\b^{l})}(A)$.

\item
Suppose that $D_{1} \in C_{(\alpha^{k},\beta^{l})}(A)$, $D_{2} \in Der_{(\alpha^{k'},\beta^{l'})}(A)$. Then for any $x, y \in A$, we have
\begin{align*}
&D_{1} \circ D_{2}(\{x, y\})\\
&= D_{1}(\{D_{2}(x), \alpha^{k'}\beta^{l'}(y)\} + \{\alpha^{k'}\beta^{l'}(x), D_{2}(y)\})\\
&= \{D_{1}(D_{2}(x)), \alpha^{k + k'}\beta^{l+l'}(y)\} + \{\alpha^{k + k'}\beta^{l+l'}(x), D_{1}(D_{2}(y))\},
\end{align*}
\begin{align*}
&D_{1} \circ D_{2}(\mu(x, y))\\
&= D_{1}(\mu(D_{2}(x), \alpha^{k'}\beta^{l'}(y)) + \mu(\alpha^{k'}\beta^{l'}(x), D_{2}(y)))\\
&= \mu(D_{1}(D_{2}(x)), \alpha^{k + k'}\beta^{l+l'}(y)) + \mu(\alpha^{k + k'}\beta^{l+l'}(x), D_{1}(D_{2}(y))),
\end{align*}
which implies that $D_{1} \circ D_{2} \in Der_{(\alpha^{k + k'},\beta^{l+l'})}(A)$. So $C(A) \circ Der(A) \subseteq Der(A)$.
\end{enumerate}
\end{proof}
\medskip
\begin{thm}
Let $(A,\{\cdot,\cdot\},\mu, \a,\b)$  be a BiHom-Poisson algebra, $\a$ and $\b$  surjections,
 then $[C(A), QC(A)]\subseteq End(A,Z(A))$.
  Moreover, if $Z(A)=\{0\}$, then  $[C(A), QC(A)]=\{0\}$.
\end{thm}

\begin{proof}
For any $D_1\in C_{(\a^{k},\b^{l})}(A), D_2\in QC_{(\a^{k'},\b^{l'})}(A)$ and $x,y\in A$, since $\a$ and $\b$ are surjections, there exist $y'\in A$ such that $y=\a^{k+k'}\b^{l+l'}(y')$, we have
\begin{eqnarray*}
\{[D_1, D_2](x), y\}&=&\{\{D_1, D_2\}(x), \a^{k+k'}\b^{l+l'}(y')\}\\
&=& \{D_1D_2(x), \a^{k+k'}\b^{l+l'}(y')\}-\{D_2D_1(x), \a^{k+k'}\b^{l+l'}(y')\}\\
&=& D_1(\{D_2(x), \a^{k'}\b^{l'}(y')\})-\{\a^{k'}\b^{l'}D_1(x), D_2\a^{k}\b^{l}(y')\}\\
&=& D_1(\{D_2(x), \a^{k'}\b^{l'}(y')\})-D_1(\{D_2(x), \a^{k'}\b^{l'}(y')\})\\
&=& 0,
\end{eqnarray*}
and
\begin{eqnarray*}
\mu([D_1, D_2](x), y)&=&\mu([D_1, D_2](x), \a^{k+k'}\b^{l+l'}(y'))\\
&=& \mu(D_1D_2(x), \a^{k+k'}\b^{l+l'}(y'))-\mu(D_2D_1(x), \a^{k+k'}\b^{l+l'}(y'))\\
&=& D_1(\mu(D_2(x), \a^{k'}\b^{l'}(y')))-\mu(\a^{k'}\b^{l'}D_1(x), D_2\a^{k}\b^{l}(y'))\\
&=& D_1(\mu(D_2(x), \a^{k'}\b^{l'}(y')))-D_1(\mu(D_2(x), \a^{k'}\b^{l'}(y')))\\
&=& 0.
\end{eqnarray*}
So $[D_1,D_2](x)\in Z(A)$ and therefore $[C(A), QC(A)]\subseteq End(A,Z(A))$.
 Moreover, if $Z(A)=\{0\}$, then  it is easy to see that $[C(A), QC(A)]=\{0\}$.
\end{proof}
\section{BiHom-Poisson modules}
First, let recall the following.
\begin{defn}
Let $(A,\{\cdot,\cdot\}_A, \mu_A )$ be a Poisson algebra. Then a left Poisson module structure on a
left $A$-module over $A$ is  linear maps
$\mu_M,\{\cdot,\cdot\}_M: A\otimes M\longrightarrow M$
such that
\begin{eqnarray}
\mu_M(a,\mu_M(b,m))=\mu_M(\mu_A(a,b),m)\label{wid}\\
\{\{a,b\}_A,m\}_M=\{a,\{b,m\}_M \}_M-\{b,\{a,m\}_M\}_M \\
\{a,\mu_M(b,m)\}_M=\mu_M(\{a, b\}_A,m)+ \mu_M(b,\{a, m\}_M) \\
\{\mu_A(a,b), m\}_M =\mu_M(a,\{b, m\}_M)+\mu_M(b,\{a, m\}_M)
\end{eqnarray}
for any $a, b\in A$ and $m\in M.$
\end{defn}
\begin{rmk}
 In \cite{LWZ}, Poisson algebras are defined without the associativity assumption and then, left Poisson modules are defined without the identity (\ref{wid}).
In a similar way, basing on the definition above, one can defined a right Poisson module.
\end{rmk}
\begin{defn}
Let $(A, \{\cdot,\cdot\},\mu,\alpha,\beta)$ be a BiHom-Poisson algebra.
\begin{enumerate}
\item A left BiHom-Poisson $A$-module is a BiHom-module $(V,\phi,\psi)$  with  structure maps  $\lambda:A\otimes V\longrightarrow V$  and  $\rho:A\otimes V\longrightarrow V$ such that the following equalities hold:
\begin{eqnarray}
\lambda(\alpha(x),\lambda(y,v))&=&\lambda(\mu(x,y),\psi(v))\label{lbhm1}\\
\rho(\{\beta(x),y \},\psi(v))&=&\rho(\alpha\beta(x),\rho(y,v))-\rho(\beta(y),\rho(\alpha(x),v))\label{lbhm2}\\
\rho(\alpha\beta(x),\lambda(y,v))&=& \lambda(\{\beta(x),y\},\psi(v))+\lambda(\beta(y),\rho(\alpha(x),v))\label{lbhm3}\\
\rho(\mu(\beta(x),y),\psi(v))&=&\lambda(\alpha\beta(x),\rho(y,v))+\lambda(\beta(y),\rho(\alpha(x),v))
\label{lbhm4}
\end{eqnarray}
\item A right BiHom-Poisson $A$-module is a BiHom-module $(V,\phi,\psi)$  with  structure maps  $\wedge:V\otimes A\longrightarrow V$  and  $\delta:V\otimes A\longrightarrow V$ such that the following equalities hold:
\begin{eqnarray}
\wedge(\wedge(v,x),\beta(y))&=&\wedge(\phi(v),\mu(x,y))\label{ebhm1}\\
\delta(\phi(v),\{x,\alpha(y)\})&=&\delta(\delta(v,x),\alpha\beta(y))-\delta(\delta(v,\beta(x)),\alpha(y))\label{rbhm2}\\
\delta(\wedge(v,x),\alpha\beta(y))&=&\delta(\phi(v),\{x,\alpha(y)\})+
\wedge(\delta(v,\beta(x)),\alpha(y))\label{rbhm3}\\
\delta(\phi(v),\mu(x,\alpha(y)))&=&\wedge(\delta(v,x),\alpha\beta(y))+\wedge(
\delta(v,\beta(x)),\alpha(y))\label{rbhm4}
\end{eqnarray}
\end{enumerate}
\end{defn}
\begin{rmk}
\begin{enumerate}
\item A left BiHom-Poisson $A$-module is a BiHom-module $(V,\phi,\psi)$  with  structure maps  $\lambda:A\otimes V\longrightarrow V$  and  $\rho:A\otimes V\longrightarrow V$ such that $(V,\phi,\psi,\lambda)$ is a left Bihom-associative $A$-module, $(V,\phi,\psi,\rho)$ is a left BiHom-Lie $A$-module 
\cite{hadimi}\cite{yong} and (\ref{lbhm3}) and (\ref{lbhm4}) hold.  Similarly a right BiHom-Poisson $A$-module is a BiHom-module $(V,\phi,\psi)$  with  structure maps  $\wedge:V\otimes A\longrightarrow V$  and  $\delta:V\otimes A\longrightarrow V$ such that $(V,\phi,\psi,\wedge)$ is a right Bihom-associative $A$-module, $(V,\phi,\psi,\delta)$ is a right BiHom-Lie $A$-module and (\ref{rbhm3}) and (\ref{rbhm4}) hold.
\item If $\alpha=\beta=Id$ and $\phi=\psi=Id$, we recover a left (respectively a right) Poisson module. Thus if $(A,\{\cdot,\cdot\}, \mu)$ is a Poisson algebra and $V$ is a left Poisson $A$-module  in the usual sense, then $(V,Id_{V},Id_{V})$ is a left BiHom-Poisson $\mathbb{A}$-module where $\mathbb{A}=(A, \{\cdot,\cdot\}, \mu,  Id_{A},
Id_{A})$ is a BiHom-Poisson algebra.
\end{enumerate}
\end{rmk}
The following results allow to give some examples of left BiHom-Poisson $A$-modules.
\begin{prop}\label{prop1}
Let $(A,\{\cdot,\cdot\},\mu,\alpha,\beta)$ be a regular BiHom-Poisson algebra. Then $(A,\alpha,\beta)$ is a left BiHom-Poisson $A$-module where the structure maps are $\lambda(a,b)=\mu(a,b)$ and $\rho(a,b)=\{a,b\}$. More generally, if $B$ is a left  BiHom-ideal of $(A,\{\cdot,\cdot\},\mu,\alpha,\beta)$, then $(B,\alpha,\beta)$ is a left BiHom-Poisson
$A$-module where the structure maps are $\lambda(a,x)=\mu(a,x),\ \rho(x,a)=\{x,a\}$ for all $x\in B$ and $(a,b)\in A^{\times 2}$.
\end{prop}
\begin{proof}
The fact that $\lambda$ and $\rho$ are structure maps follows from the multiplicativity of $\alpha$ and  $\beta$ with respect to $\mu$ and $\{\cdot,\cdot\}.$
Next, observe that from the BiHom-commutativity of $\mu$ and the Bihom-skew-symmetry of
$\{\cdot,\cdot\}$ that  $\mu(x,y)=\mu(\alpha^{-1}\beta(y),\alpha\beta^{-1}(x))$ and
$\{x,y\}=-\{\alpha^{-1}\beta(y),\alpha\beta^{-1}(x)\}$ for all $x,y\in A.$ Now, pick $(x,y,v)\in A^{\times 3}$ then, we have by the BiHom-associativity
\begin{eqnarray}
&&\lambda(\alpha(x),\lambda(y,v))=\mu(\alpha(x),\mu(y,v))=\mu(\mu(x,y),\beta(v))=
\lambda(\mu(x,y),\beta(v))\nonumber
\end{eqnarray}
Next, compute (\ref{lbhm2}) using the BiHom-Jacobi identity in the third line, as follows
\begin{eqnarray}
&&\rho(\{\beta(x),y\},\beta(v))=\{\{\beta(x),y\},\beta(v)\}=
\{\alpha^{-1}\beta(\beta(v)),\alpha\beta^{-1}\{\alpha^{-1}\beta(y),\alpha\beta^{-1}(\beta(x))\}\}\nonumber\\
&&=\{\beta^2(\alpha^{-1}(v)),\{y,\alpha^2\beta^{-1}(x)\}\}=
\{\beta^2(\alpha^{-1}(v)),\{\beta(\beta^{-1}(y)),\alpha(\alpha\beta^{-1}(x))\}\}
\nonumber\\
&&=-\{\beta^2(\beta^{-1}(y)),\{\beta(\alpha\beta^{-1}(x)),\alpha(\alpha^{-1}(v))\}\}
-\{\beta^2(\alpha\beta^{-1}(x)),\{\beta(\alpha^{-1}(v)),\alpha(\beta^{-1}(y))\}\}
\nonumber\\
&&=-\{\beta(y),\{\alpha(x),v)\}\}
-\{\alpha\beta(x),\{\beta\alpha^{-1}(v),\alpha\beta^{-1}(y)\}\}=-\{\beta(y),\{\alpha(x),v)\}\}\nonumber\\
&&+\{\alpha\beta(x),\{y,v\}\}=-\rho(\beta(y),\rho(\alpha(x),v))+\rho(\alpha\beta(x),\rho(y,v))\nonumber
\end{eqnarray}
Similarly, using (\ref{BHL}), we compute
\begin{eqnarray}
&&\rho(\alpha\beta(x),\lambda(y,v))=\{\alpha\beta(x),\mu(y,v)\}
=\mu(\{\beta(x),y\},\beta(v))+\mu(\beta(y),\{\alpha(x),v\})\nonumber\\
&&=\lambda(\{\beta(x),y\},\beta(v))+\lambda(\beta(y),\rho(\alpha(x),v))
 \mbox{\ which is (\ref{lbhm3})} \nonumber
\end{eqnarray}
Finally, we obtain (\ref{lbhm4}) as follows
\begin{eqnarray}
&& \rho(\mu(\beta(x),y),\beta(v))=\{\mu(\beta(x),y),\beta(v)\}=
-\{\alpha^{-1}\beta(\beta(v)),\alpha\beta^{-1}\mu(\beta(x),y)\}\nonumber\\
&&=-\{\alpha\beta(\alpha^{-2}\beta(v)),\mu(\alpha(x),\alpha\beta^{-1}(y))\}
=-\mu(\{\beta(\alpha^{-2}\beta(v)),\alpha(x)\},\beta(\beta^{-1}\alpha(y)))
\nonumber\\
&&-\mu(\beta(\alpha(x)),\{\alpha(\alpha^{-2}\beta(v)),\alpha\beta^{-1}(y)\})
\mbox{\  ( by (\ref{BHL})  )} \nonumber\\
&&=-\mu(\{\beta^2\alpha^{-2}(v),\alpha(x)\},\alpha(y))
-\mu(\beta\alpha(x),\{\alpha^{-1}\beta(v),\alpha\beta^{-1}(y)\})\nonumber\\
&&=-\mu(\alpha^{-1}\beta(\alpha(y)),\alpha\beta^{-1}\{\beta^2\alpha^{-2}(v),\alpha(x)\})+\mu(\alpha\beta(x),\{y,v\})\nonumber\\
&&=-\mu(\beta(y),\{\beta\alpha^{-1}(v),\alpha\beta^{-1}(\alpha(x))\})+\mu(\alpha\beta(x),\{y,v\})\nonumber\\
&&=\mu(\beta(y),\{\alpha(x),v\})+\mu(\alpha\beta(x),\{y,v\})
=\lambda(\beta(y),\rho(\alpha(x),v))+\lambda(\alpha\beta(x),\rho(y,v)).
\nonumber
\end{eqnarray}
Hence $(A,\alpha,\beta)$ is a left BiHom-Poisson $A$-module.
Similarly, we prove that more generally, any two-sided BiHom-ideal $(B,\alpha,\beta)$  of $(A,\{\cdot,\cdot\},\mu,\alpha,\beta)$ is a left BiHom-Poisson
$A$-module.
\end{proof}
\begin{rmk}
The analogous of Proposition \ref{prop1} can be proved for right BiHom-Poisson algebras.
\end{rmk}
More generally, we prove:
\begin{prop}
If $f:(A,\{\cdot,\cdot\}_A,\mu_A, \alpha_{A},\beta_{A})\longrightarrow(B,\{\cdot,\cdot\}_A, \mu_B,\alpha_{B},\beta_{B})$ is a morphism of BiHom-Poisson algebras and $\alpha_B$ and $\beta_B$ are invertible then, $(B,\alpha_{B},\beta_{B})$ becomes a
left BiHom-Poisson $A$-module via $f$, i.e, the structure maps are defined as $\lambda(a,b)=\mu_{B}(f(a),b)$ and $\rho(a,b)=\{f(a),b\}_B$ for all
$(a,b)\in A\times B$.
\end{prop}
\begin{proof}
The fact that $\lambda$ and $\rho$ are structure maps follows from the multiplicativity of $\alpha_B$ and  $\beta_B$ with respect to $\mu_B$ and $\{\cdot,\cdot\}_B.$
Next, observe that from the BiHom-commutativity of $\mu_B$ and the Bihom-skew-symmetry of
$\{\cdot,\cdot\}_B$ that $\mu_B(b_1,b_2)=\mu_B(\alpha^{-1}\beta(b_2),\alpha\beta^{-1}(b_1))$ and
$\{b_1,b_2\}_B=-\{\alpha_B^{-1}\beta_B(b_2),\alpha_B\beta_B^{-1}(b_1)\}_B$ for all $b_1,b_2\in B.$ Now,
pick $(x,y)\in A^{\times 2}$ and $v\in B$ then sine $f$ is a morphism
of BiHom-algebras, we have by the BiHom-associativity in $B$
\begin{eqnarray}
&&\lambda(\alpha_A(x),\lambda(y,v))=\mu_B(f\alpha_A(x),\mu_B(f(y),v))=
\mu_B(\alpha_B(f(x)),\mu_B(f(y),v))\nonumber\\
&&=\mu_B(f\mu_B(x,y),\beta_B(v))=
\lambda(\mu_B(x,y),\beta_B(v))\nonumber
\end{eqnarray}
Next, compute (\ref{lbhm2}) using the BiHom-Jacobi identity in the third line, as follows
\begin{eqnarray}
&&\rho(\{\beta_A(x),y\}_A,\beta_B(v))=\{f\{\beta(x),y\}_A,\beta(v)\}_B=
\{\{\beta_B(f(x)),f(y)\}_B,\beta_B(v)\}_B\nonumber\\
&&=\{\alpha_B^{-1}\beta_B(\beta_B(v)),\alpha_B\beta_B^{-1}\{\alpha_B^{-1}\beta_B f(y),\alpha_B\beta_B^{-1}(\beta_B f(x))\}_B\}_B\nonumber\\
&&=\{\beta_B^2(\alpha_B^{-1}(v)),\{f(y),\alpha_B^2\beta_B^{-1}f(x)\}_B\}_B\nonumber\\
&&=
\{\beta_B^2(\alpha_B^{-1}f(v)),\{\beta_B(\beta_B^{-1}f(y)),\alpha_B(\alpha_B\beta_B^{-1}f(x))\}_B\}_B
\nonumber\\
&&=-\{\beta_B^2(\beta_B^{-1}f(y)),\{\beta_B(\alpha_B\beta_B^{-1}f(x)),\alpha_B(\alpha_B^{-1}(v))\}_B\}_B\nonumber\\
&&-\{\beta_B^2(\alpha_B\beta_B^{-1}f(x)),\{\beta_B(\alpha_B^{-1}(v)),\alpha_B(\beta_B^{-1}f(y))\}_B\}_B
\nonumber\\
&&=-\{\beta_B(f(y)),\{\alpha_B(f(x)),v)\}_B\}_B
-\{\alpha_B\beta_B(f(x)),\{\beta_B\alpha_B^{-1}(v),\alpha_B\beta_B^{-1}(f(y))\}_B\}_B\nonumber\\
&&=-\{\beta_B(f(y)),\{\alpha_B(f(x)),v)\}_B\}_B+\{\alpha_B\beta_B(f(x)),\{f(y),v\}_B\}_B\nonumber\\
&&=-\{f(\beta_A(y)),\{f(\alpha_A(x)),v)\}_B\}_B+\{f(\alpha_A\beta_A(x)),\{f(y),v\}_B\}_B\nonumber\\
&&=-\rho(\beta(y),\rho(\alpha(x),v))+\rho(\alpha\beta(x),\rho(y,v))\nonumber
\end{eqnarray}
Similarly, using (\ref{BHL}) for $B$ and $f$ is a morphism, we compute
\begin{eqnarray}
&&\rho(\alpha_A\beta_A(x),\lambda(y,v))=\{f(\alpha_A\beta_A(x)),\mu_B(f(y),v)\}_B
=\{\alpha_B\beta_B(f(x)),\mu_B(f(y),v)\}_B\nonumber\\
&&=\mu_B(\{\beta_B(f(x)),f(y)\},\beta_B(v))+\mu_B(\beta_B(f(y)),\{\alpha_B(f(x)),v\}_B)\nonumber\\
&&=\mu_B(f\{\beta_A(x),y\}_A,\beta_B(v))+\mu_B(f(\beta_A(y)),\{f(\alpha_B(x)),v\}_B)\nonumber\\
&&=\lambda(\{\beta_A(x),y\}_A,\beta_B(v))+\lambda(\beta_A(y),\rho(\alpha_A(x),v))
 \mbox{\ which is (\ref{lbhm3})} \nonumber
\end{eqnarray}
Finally, using $f$ is a morphism we obtain (\ref{lbhm4}) as follows
\begin{eqnarray}
&& \rho(\mu_A(\beta_A(x),y),\beta_B(v))=\{f\mu_A(\beta_A(x),y),\beta_B(v)\}_B\nonumber\\
&&=-\{\alpha_B^{-1}\beta_B(\beta_B(v)),\alpha_B\beta_B^{-1}f\mu_A(\beta_A(x),y)\}_B\nonumber\\
&&=-\{\alpha_B\beta_B(\alpha_B^{-2}\beta_B(v)),\mu_B(\alpha_B(f(x)),\alpha_B\beta_B^{-1}(f(y)))\}_B\nonumber\\
&&=-\mu_B(\{\beta_B(\alpha_B^{-2}\beta_B(v)),\alpha_B(f(x))\}_B,\beta_B(\beta_B^{-1}\alpha_B f(y)))
\nonumber\\
&&-\mu_B(\beta_B(\alpha_B f(x)),\{\alpha_B(\alpha_B^{-2}\beta_B(v)),\alpha_B\beta_B^{-1}(f(y))\})
\mbox{\  ( by (\ref{BHL})  in $B$ )} \nonumber\\
&&=-\mu_B(\{\beta_B^2\alpha_B^{-2}(v),\alpha_B(f(x))\}_B,\alpha_B(f(y)))\nonumber\\
&&-\mu_B(\beta_B\alpha_B(f(x)),\{\alpha_B^{-1}\beta_B(v),\alpha_B\beta_B^{-1}(f(y))\}_B)\nonumber\\
&&=-\mu_B(\alpha_B^{-1}\beta_B(\alpha_B f(y)),\alpha_B\beta_B^{-1}\{\beta_B^2\alpha_B^{-2}(v),\alpha_B(f(x))\}_B)+\mu_B(\alpha_B\beta_B(f(x)),\{f(y),v\}_B)\nonumber\\
&&=-\mu_B(\beta_B(f(y)),\{\beta_B\alpha_B^{-1}(v),\alpha_B\beta_B^{-1}(\alpha_B f(x))\})+\mu_B(\alpha_B\beta_B(f(x)),\{f(y),v\}_B)\nonumber\\
&&=\mu_B(\beta_B(f(y)),\{\alpha_B(f(x)),v\}_B)+\mu_B(\alpha_B\beta_B(f(x)),\{f(y),v\}_B)\nonumber\\
&&=\mu_B(f(\beta_A(y)),\{f(\alpha_A(x)),v\}_B)+\mu_B(f(\alpha_A\beta_A(x)),\{f(y),v\}_B)\nonumber\\
&&=\lambda(\beta_A(y),\rho(\alpha_A(x),v))+\lambda(\alpha_A\beta_A(x),\rho(y,v))
\nonumber
\end{eqnarray}
Hence $(B,\alpha_{B},\beta_{B})$ is a
left BiHom-Poisson $A$-module.
\end{proof}
Similarly, we can prove:
\begin{prop}
If $f:(A,\{\cdot,\cdot\}_A, \mu_A, \alpha_{A},\beta_{A})\longrightarrow(B,\{\cdot,\cdot\}_B, \mu_B, \alpha_{B},\beta_{B})$ is a morphism of BiHom-Poisson algebras and $\alpha_B$ and $\beta_B$ are invertible then, $(B,\alpha_{B},\beta_{B})$ becomes a
right BiHom-Poisson $A$-module via $f$, i.e, the structure maps are defined as $\lambda(b,a)=\mu_{B}(b,f(a))$ and $\rho(b,a)=\{b,f(a)\}_B$ for all
$(a,b)\in A\times B$.
\end{prop}
As the case of BiHom-alternative and BiHom-Jordan algebras \cite{LA}, in order to give another example of left BiHom-Poisson modules, let us consider the following:
\begin{defn} An abelian extension of BiHom-Poisson algebras is a short exact sequence of BiHom-Poisson algebras
$$0\longrightarrow(V,\alpha_V,\beta_V)\stackrel{\mbox{i}}{\longrightarrow} (A,\{\cdot,\cdot\}_A, \mu_A, \alpha_A,\beta_A)\stackrel{\mbox{$\pi$}}{\longrightarrow} (B,\{\cdot,\cdot\}_B, \mu_B, \alpha_B,\beta_B)\longrightarrow 0$$
 where $(V,\alpha_V,\beta_B)$ is a trivial BiHom-Poisson algebra, $i$ and $\pi$ are morphisms of BiHom-algebras. Furthermore,  if there exists a morphism
$s :(B,\{\cdot,\cdot\}_B, \mu_B, \alpha_B,\beta_B)\longrightarrow\\ (A,\{\cdot,\cdot\}_A, \mu_A, \alpha_A, \beta_A)$ such that $\pi\circ s = id_B$ then the abelian extension is said to be split and $s$ is called a section of $\pi.$
\end{defn}
\begin{ex}
Given an abelian extension as in the previous definition,  the BiHom-module
$(V,\alpha_V, \beta_V)$ inherits a structure of a left BiHom-Poisson $B$-module and the actions of the BiHom-algebra $(B,\{\cdot,\cdot\}_B,\mu_B, \alpha_B,\beta_B)$  on $V$ are as follows. For any $x\in B,$ there exist $\tilde{x}\in A$ such that $x=\pi(\tilde{x}).$ Let $x$ acts on $v\in V$ by $\lambda(x, v):=\mu_A(\tilde{x},i(v))$ and $\rho(x, v):=\{\tilde{x},i(v)\}_A.$
These are well-defined, as another lift $\tilde{x'}$
of $x$ is written $\tilde{x'}=\tilde{x}+v'$ for some $v'\in V$ and thus
$\lambda(x,v)=\mu_A(\tilde{x},i(v))=\mu_A(\tilde{x'},i(v))$ and $\rho(x,v)=\{\tilde{x},i(v)\}_A=\{\tilde{x'},i(v)\}_A$ because $V$ is trivial. The actions property follow from the BiHom-Poisson identities.
In case these actions of $B$ on $V$  are trivial, one speaks of a central extension.
\end{ex}
The next result allow to construct a sequence of left BiHom-Poisson modules from a given one.
\begin{prop}\label{HJB-HJB}
Let $(A,\{\cdot,\cdot\},\mu,\alpha,\beta)$ be a BiHom-Poisson algebra and  $V_{\phi,\psi}=(V,\phi, \psi)$  be a left BiHom-Poisson $A$-module with the structure maps $\lambda$ and $\rho$.  Then for each $n, m\in\mathbb{N},$ the maps
\begin{eqnarray}
\lambda^{(n,m)}=\lambda\circ(\alpha^n\beta^m\otimes Id_V)\label{nmj1}\\
\rho^{(n,m)}=\rho\circ(\alpha^n\beta^m\otimes Id_V)\label{nmj2}
\end{eqnarray}
give the BiHom-module $(V,\phi, \psi)$ the structure of a left BiHom-Poisson $A$-module that we denote
by $V_{\phi,\psi}^{(n,m)}.$
\end{prop}
\begin{proof}
Since the structure map $\lambda$ is a morphism of BiHom-modules, we get:
\begin{eqnarray}
&&\phi\lambda^{(n,m)}=\phi\rho_l\circ(\alpha^n\beta^m\otimes Id_V)
\mbox{ \textsl{( by (\ref{nmj1}) )}}\nonumber\\
&=&\lambda\circ(\alpha^{n+1}\beta^m\otimes\phi)
=\lambda\circ(\alpha^n\beta^m\otimes Id_V)\circ(\alpha\otimes\phi)
=\lambda^{(n,m)}\circ(\alpha\otimes\phi)\nonumber
\end{eqnarray}
 Similarly, we get that $\psi\lambda^{(n,m)}=\lambda^{(n,m)}\circ(\beta\otimes\psi),\
 \phi\rho^{(n,m)}=\rho^{(n,m)}\circ(\phi\otimes\alpha)$ and $\psi\rho^{(n,m)}=\rho^{(n,m)}\circ(\psi\otimes\beta).$  Thus $\rho_l^{(n)}$ and $\rho_r^{(n)}$ are morphisms of BiHom-modules. First, pick $(x,y)\in A^{\times 2}$ and $v\in V,$ then using (\ref{lbhm1}) in the second line for $V_{\phi,\psi},$ we get
 \begin{eqnarray}
&& \lambda^{(n,m)}(\alpha(x),\lambda^{(n,m)}(y,v))=\lambda(\alpha^{n+1}\beta^m(x),
 \lambda(\alpha^n\beta^m(y),v))\nonumber\\
 &&=\lambda(\mu(\alpha^n\beta^m(x),\alpha^n\beta^m(y)),\psi(v))=\lambda^{(n,m)}
 (\mu(x,y),\psi(v))\nonumber
 \end{eqnarray}
 Secondly, we compute
 \begin{eqnarray}
 &&\rho^{(n,m)}(\{\beta(x),y\},\psi(v))=\rho(\{\alpha^n\beta^{m+1}(x),\beta(y)\},
 \psi(v))\nonumber\\
 &&=\rho(\alpha\beta\alpha^n\beta^m(x),\rho(\alpha^n\beta^m(y),v))-\rho(\beta\alpha^n\beta^m(y),\rho(\alpha\alpha^n\beta^m(x),v)) \mbox{ ( by (\ref{lbhm2}) in  $V_{\phi,\psi}$ )} \nonumber\\
 &&=\lambda^{(n,m)}(\alpha\beta(x),\lambda^{(n,m)}(y,v))-\lambda^{(n,m)}(\beta(y),\lambda^{(n,m)}(\alpha(y),v))\nonumber
 \end{eqnarray}
 Next, we obtain
 \begin{eqnarray}
&&\rho^{(n,m)}(\beta\alpha(x),\lambda^{(n,m)}(y,v))=\rho(\alpha\beta\alpha^n\beta^m(x),\lambda(\alpha^n\beta^m(y),v))\nonumber\\
&&=\lambda\{\beta\alpha^n\beta^m(x),\alpha^n\beta^m(y)\},\psi(v))+
\lambda(\beta\alpha^n\beta^(y),\rho(\alpha\alpha^n\beta^m(x),v)) \mbox{ ( by (\ref{lbhm3}) in  $V_{\phi,\psi}$ )} \nonumber\\
&&=\lambda^{(n,m)}(\{\beta(x),y\},\psi(v))+\lambda^{(n,m)}(\beta(y),
\rho^{(n,m)}(\alpha(x),v))\nonumber
 \end{eqnarray}
 Finally, we compute
 \begin{eqnarray}
 &&\rho^{(n,m)}(\mu(\beta(x),y),\psi(v))=\rho(\mu(\beta\alpha^n\beta^m(x),\alpha^n\beta^m(y)),\psi(v))\nonumber\\
 &&=\lambda(\alpha\beta\alpha^n\beta^m(x),\rho(\alpha^n\beta^m(y),v))+
 \lambda(\beta\alpha^n\beta^m(y),\rho(\alpha\alpha^n\beta^m(x),v))
 \mbox{ ( by (\ref{lbhm3}) in  $V_{\phi,\psi}$ )} \nonumber\\
 &&=\lambda^{(n,m)}(\alpha\beta(x),\rho^{(n,m)}(y,v))+\lambda^{(n,m)}(
 \beta(y), \rho^{(n,m)}(\alpha(x),v)).\nonumber
 \end{eqnarray}
 Hence, $V_{\phi,\psi}^{(n,m)}$ is a left BiHom-Poisson $A$-module.
\end{proof}
 Proposition \ref{HJB-HJB} reads for the case of right BiHom-Poisson module as:
\begin{prop}\label{HJB-HJBr}
Let $(A,\{\cdot,\cdot\},\mu,\alpha,\beta)$ be a BiHom-Poisson algebra and  $V_{\phi,\psi}=(V,\phi, \psi)$  be a right BiHom-Poisson $A$-module with the structure maps $\wedge$ and $\delta$.  Then for each $n, m\in\mathbb{N},$ the maps
\begin{eqnarray}
\wedge^{(n,m)}=\wedge\circ(Id_V\otimes\alpha^n\beta^m)\label{nmj1}\\
\delta^{(n,m)}=\delta\circ(Id_V\otimes\alpha^n\beta^m)\label{nmj2}
\end{eqnarray}
give the BiHom-module $(V,\phi, \psi)$ the structure of a right BiHom-Poisson $A$-module that we denote
by $V_{\phi,\psi}^{(n,m)}.$
\end{prop}
\begin{thm}\label{HJB-JB}
Let $(A,\{\cdot,\cdot\}, \mu)$ be a Poisson algebra, $V$ be a left Poisson $A$-module with the structure maps
$\lambda$, $\rho$ and $\alpha,\beta$ be  endomorphisms of the Jordan algebra $A$
and $\phi, \psi$ be  linear self-maps of $V$ such that $\phi\circ\lambda=\lambda\circ(\alpha\otimes\phi),$
$\phi\circ\rho=\rho\circ(\alpha\otimes\phi),$
$\psi\circ\lambda=\lambda\circ(\beta\otimes\psi)$ and
$\psi\circ\rho=\rho\circ(\beta\otimes\psi)$.
Write $A_{\alpha,\beta}$ for the BiHom-Poisson algebra $(A,\{\cdot,\cdot\}_{\alpha,\beta}=\{\cdot,\cdot\}(\alpha\otimes\beta),\mu_{\alpha,\beta}=\mu(\alpha\otimes\beta), \alpha, \beta)$ and
$V_{\phi,\psi}$ for the BiHom-module $(V,\phi,\psi).$ Then the maps:
\begin{eqnarray}
\tilde{\lambda}=\lambda\circ(\alpha\beta\otimes\psi) \mbox{ and }
\tilde{\rho}=\rho\circ(\alpha\beta\otimes\psi)
\end{eqnarray}
give the BiHom-module $V_{\phi,\psi}$ the structure of a left BiHom-Poisson $A_{\alpha,\beta}$-module.
\end{thm}
\begin{proof}
It is clear that $\tilde{\lambda}$ and $\tilde{\rho}$ are morphisms of
BiHom-modules. Netx, pick $(x,y)\in A^{\times 2}$ and $v\in V,$ then
using (\ref{lbhm1}) for  $V_{\phi,\psi}$ in the second line  we get:
 \begin{eqnarray}
&& \tilde{\lambda}(\alpha(x),\tilde{\lambda}(y,v))=\lambda(\alpha^2\beta(x),
 \lambda(\alpha\beta^2(y),\psi^2(v)))\nonumber\\
 &&=\lambda(\mu(\alpha^2\beta(x),\alpha\beta^2(y)),\psi^2(v))=\lambda(\alpha\beta\mu(\alpha(x),\beta(y)),\psi^2(v))=\tilde{\lambda}
 (\mu_{\alpha,\beta}(x,y),\psi(v))\nonumber
 \end{eqnarray}
 Secondly, we compute
 \begin{eqnarray}
 &&\tilde{\rho}(\{\beta(x),y\}_{\alpha,\beta},\psi(v))=\rho(\{\alpha^2\beta^2(x),\alpha\beta^2(y)\},
 \psi^2(v))\nonumber\\
 &&=\rho(\alpha^2\beta^2(x),\rho(\alpha\beta^2(y),\psi^2(v)))-\rho(\alpha\beta^2(y),\rho(\alpha^2\beta^2(x),\psi^2(v)) \mbox{ ( by (\ref{lbhm2}) in  $V_{\phi,\psi}$ )} \nonumber\\
 &&=\tilde{\lambda}(\alpha\beta(x),\tilde{\lambda}(y,v))-\tilde{\lambda}(\beta(y),\tilde{\lambda}(\alpha(y),v))\nonumber
 \end{eqnarray}
 Next, we obtain
 \begin{eqnarray}
&& \tilde{\rho}(\beta\alpha(x), \tilde{\lambda}(y,v))=\rho(\alpha^2\beta(x),\lambda(\alpha\beta^2(y),\psi^2(v)))\nonumber\\
&&=\lambda\{\alpha^2\beta(x),\alpha\beta^2(y)\},\psi^2(v))+
\lambda(\alpha\beta^2(y),\rho(\alpha^2\beta(x),\psi^2(v))) \mbox{ ( by (\ref{lbhm3}) in  $V_{\phi,\psi}$ )} \nonumber\\
&&= \tilde{\lambda}(\{\beta(x),y\}_{\alpha,\beta},\psi(v))+ \tilde{\lambda}(\beta(y),
 \tilde{\rho}(\alpha(x),v))\nonumber
 \end{eqnarray}
 Finally, we compute
 \begin{eqnarray}
 && \tilde{\rho}(\mu_{\alpha,\beta}(\beta(x),y),\psi(v))=\rho(\mu(\alpha^2\beta^2(x),\alpha\beta^2(y)),\psi^2(v))\nonumber\\
 &&=\lambda(\alpha^2\beta^2(x),\rho(\alpha\beta^2(y),\psi^2(v)))+
 \lambda(\alpha\beta^2(y),\rho(\alpha^2\beta^2(x),\psi^2(v)))
 \mbox{ ( by (\ref{lbhm3}) in  $V_{\phi,\psi}$ )} \nonumber\\
 &&= \tilde{\lambda}(\alpha\beta(x), \tilde{\rho}(y,v))+ \tilde{\lambda}
 (\beta(y),  \tilde{\rho}(\alpha(x),v)))\nonumber
 \end{eqnarray}
 Hence, $V_{\phi,\psi}^{(n,m)}$ is a left BiHom-Poisson $A$-module.
\end{proof}
\begin{cor}
Let $(A,\{\cdot,\cdot\}, \mu)$ be a Poisson algebra, $V$ be a left Poisson $A$-module with the structure maps
$\lambda$ and $\rho$, $\alpha,\ \beta$ be  endomorphisms of the Poisson algebra $A$ and $\phi, \psi$ be  linear self-maps of $V$ such that $\phi\circ\lambda=\lambda\circ(\alpha\otimes\phi),$
$\phi\circ\rho=\rho\circ(\alpha\otimes\phi),$
$\psi\circ\lambda=\lambda\circ(\beta\otimes\psi)$ and
$\psi\circ\rho=\rho\circ(\alpha\otimes\psi)$.\\
Write $A_{\alpha,\beta}$ for the BiHom-Poisson algebra $(A,\{\cdot,\cdot\}_{\alpha,\beta}=\{\alpha\otimes\beta\}, \mu_{\alpha,\beta}=\mu(\alpha\otimes\beta), \alpha, \beta)$ and
$V_{\phi,\psi}$ for the BiHom-module $(V,\phi,\psi).$ Then the maps:
\begin{eqnarray}
\tilde{\lambda}^{(n,m)}=\lambda\circ(\alpha^{n+1}\beta^{m+1}\otimes\psi) \mbox{ and }
\tilde{\rho}^{(n,m)}=\rho\circ(\alpha^{n+1}\beta^{m+1}\otimes\psi)
\end{eqnarray}
give the BiHom-module $V_{\phi,\psi}$ the structure of a left BiHom-Poisson $A_{\alpha,\beta}$-module for all
$n, m\in\mathbb{N}$.
\end{cor}
\begin{proof} The proof follows from Proposition \ref{HJB-HJB} and Theorem \ref{HJB-JB}.
\end{proof}
Similarly, we can prove the following result which is the analogous of
Theorem \ref{HJB-JB} for right BiHom-Poisson modules.
\begin{thm}\label{HJB-JB}
Let $(A,\{\cdot,\cdot\}, \mu)$ be a Poisson algebra, $V$ be a right Poisson $A$-module with the structure maps
$\lambda$, $\rho$ and $\alpha,\beta$ be  endomorphisms of the Jordan algebra $A$
and $\phi, \psi$ be  linear self-maps of $V$ such that $\phi\circ\lambda=\lambda\circ(\alpha\otimes\phi),$
$\phi\circ\rho=\rho\circ(\alpha\otimes\phi),$
$\psi\circ\lambda=\lambda\circ(\beta\otimes\psi)$ and
$\psi\circ\rho=\rho\circ(\beta\otimes\psi)$.
Write $A_{\alpha,\beta}$ for the BiHom-Poisson algebra $(A,\{\cdot,\cdot\}_{\alpha,\beta}=\{\cdot,\cdot\}(\alpha\otimes\beta), \mu_{\alpha,\beta}=\mu(\alpha\otimes\beta), \alpha, \beta)$ and
$V_{\phi,\psi}$ for the BiHom-module $(V,\phi,\psi).$ Then the maps:
\begin{eqnarray}
\tilde{\wedge}=\wedge\circ(\phi\otimes \alpha\beta) \mbox{ and }
\tilde{\delta}=\rho\circ(\phi\otimes\alpha\beta)
\end{eqnarray}
give the BiHom-module $V_{\phi,\psi}$ the structure of a right BiHom-Poisson $A_{\alpha,\beta}$-module.
\end{thm}
In the case of Poisson algebras, we can form semidirect products when given a left (or a right) module. Similarly, we have
\begin{thm}
Let $(A, \{\cdot,\cdot\}, \mu, \alpha, \beta)$ be a BiHom-Poisson algebra and  $(V,\phi, \psi)$  be a left $A$-module with the structure maps $\lambda$ and $\rho$. Then $(A\oplus V,[\dot,\dot],\ast, \tilde{\alpha}, \tilde{\beta})$ is
a BiHom-Poisson algebra where
$\ast,[\cdot,\cdot]: (A\oplus V)^{\otimes 2}\longrightarrow A\oplus V,~$
$(a+u)\ast(b+v):=\mu(a,b)+\lambda(a,v)+\lambda(\alpha^{-1}\beta(b),\psi^{-1}\phi(u)),$  $[a+u,b+v]:=\{a,b\}+\rho(a,v)-\rho(\alpha^{-1}\beta(b),\psi^{-1}\phi(u))$  and
$\tilde{\alpha}, \tilde{\beta}: A\oplus V\longrightarrow A\oplus V,$ $\tilde{\alpha}(a+u):=\alpha(a)+\phi(u)$  and $\tilde{\beta}(a+u):=\beta(a)+\psi(u)$ called the
semidirect product of the BiHom-Poisson $(A,\mu, \{\cdot,\cdot\},\alpha, \beta)$ and  $(V,\phi, \psi).$
\end{thm}
\begin{proof} Clearly, $\tilde{\alpha}$ and $\tilde{\beta}$ are multiplicative
with respect to $\ast$ and $[\cdot,\cdot].$ Next
\begin{eqnarray}
&&\tilde{\beta}(a+u)\ast\tilde{\alpha}(b+v)=(\beta(a)+\psi(u))\ast(\alpha(b)+\phi(v))\nonumber\\
&&=\mu(\beta(a),\alpha(b))+\lambda(\beta(a),\phi(v))+\lambda(\beta(b),\phi(u)\nonumber\\
&&=\mu(\beta(b),\alpha(a))+\lambda(\beta(b),\phi(u)+\lambda(\beta(a),\phi(v)) \mbox{ ( by the BiHom-commutativity of $\mu$ )}\nonumber\\
&&=\tilde{\beta}(b+v)\ast\tilde{\alpha}(u+v).
\end{eqnarray}
Next, pick  $(a,b,c)\in A^{\times 2}$ and $(u,v,w)\in V^{\times 2},$ then
\begin{eqnarray}
&&\Big((a+u)\ast(b+v)\Big)\ast\bar{\beta}(c+w)=
\Big(\mu(a,b)+\lambda(a,v)+\lambda(\alpha^{-1}\beta(b),\psi^{-1}\phi(u))\Big)
\ast(\beta(c)+\psi(w))\nonumber\\
&&=\mu(\mu(a,b),\beta(c))+\lambda(\mu(a,b),\psi(w))+\lambda(\alpha^{-1}\beta^2(c),
\lambda(\beta^{-1}\alpha(a),\phi\psi^{-1}(v))\nonumber\\
&&+\lambda(\alpha^{-1}\beta^2(c),\lambda(b,(\psi^{-1}\phi)^2(u)))
\mbox{  ( using $\lambda$ is a morphism )}\nonumber\\
&&=\mu(\alpha(a),\mu(b,c))+\lambda(\mu(a,b),\psi(w))+\lambda(\mu((\alpha^{-1}\beta)^2(c),\beta^{-1}\alpha(a)),\phi(v))
\nonumber\\
&&+\lambda(\mu((\alpha^{-1}\beta)^2(c),b),\psi^{-1}\phi^2(u))
\mbox{  ( using  BiHom-associativity  and (\ref{lbhm1}) )}\\
&&=\mu(\alpha(a),\mu(b,c))+\lambda(\mu(a,b),\psi(w))+\lambda(\mu(a,\beta\alpha^{-1}(c)),\phi(v))
\nonumber\\
&&+\lambda(\mu(\alpha^{-1}\beta(b),\alpha^{-1}\beta(c),\psi^{-1}\phi^2(u))
\mbox{  ( using  BiHom-commutativity )}\nonumber
\end{eqnarray}
Similarly, we prove that
\begin{eqnarray}
&&\bar{\alpha}(a+u)\ast\Big((b+v)\ast(c+w)\Big)
=\mu(\alpha(a),\mu(b,c))+\lambda(\mu(a,b),\psi(w))+\lambda(\mu(a,\beta\alpha^{-1}(c)),\phi(v))\nonumber\\
&&+\lambda(\mu(\alpha^{-1}\beta(b),\alpha^{-1}\beta(c),\psi^{-1}\phi^2(u))
\nonumber
\end{eqnarray}
Hence $(A\oplus V, \ast,\bar{\alpha}, \bar{\beta})$ is a BiHom-commutative BiHom-associative algebra. Now, observe that $(A\oplus V,[\cdot,\cdot],\bar{\alpha},\bar{\beta})$ is a BiHom-Lie algebra (Proposition 4.9, \cite{GRAZIANI} ).
 Finally, let $(a,b,c)\in A^{\times 2}$ and $(u,v,w)\in V^{\times 2}.$ Then
 \begin{eqnarray}
 &&[\bar{\alpha}\bar{\beta}(a+u),(b+v)\ast (c+w)]=[\alpha\beta(a)+\phi\psi(u),
 \mu(b,c)+\lambda(b,w)+\lambda(\alpha^{-1}\beta(),\psi^{-1}\phi(v)]\nonumber\\
 &&=\{\alpha\beta(a),\mu(b,c)\}+\rho(\alpha\beta(a),\lambda(b,w))+\rho(\alpha\beta(a),\lambda(\alpha^{-1}\beta(c),\psi^{-1}\phi(v)))
 \nonumber\\
 &&-\rho(\mu(\beta\alpha^{-1}(b),\beta\alpha^{-1}(c)),\psi\psi^{-1}\phi^2(u))
 =\mu(\{\beta(a),b\},\beta(c))\nonumber\\
 &&+\mu(\beta(c),\{\alpha(a),c\})+\lambda(\{\beta(a),b\},\psi(w))
 +\lambda(\beta(b),\rho(\alpha(a),w))+\lambda(\{\beta(a),\alpha^{-1}\beta(c)\},\phi(v))\nonumber\\
 &&+\lambda(\beta^2\alpha^{-1}(c),\rho(\alpha(a),\psi^{-1}\phi(v)))
 -\lambda(\beta(b),\rho(\beta\alpha^{-1}(c),\psi^{-1}\phi^2(u)))\nonumber\\
 &&-\lambda(\beta^2\alpha^{-1}(c),\rho(b,\psi^{-1}\phi^2(u)))
 \mbox{ ( by (\ref{BHL}}), (\ref{lbhm3}), (\ref{lbhm4}) )\nonumber\\
&& =\Big(\mu(\{\beta(a),b\},\beta(c))+ \lambda(\{\beta(a),b\},\psi(w))+
 \lambda(\beta^2\alpha^{-1}(c),\rho(\alpha(a),\psi^{-1}\phi(v)))\nonumber\\
 &&-\lambda(\beta^2\alpha^{-1}(c),\rho(b,\psi^{-1}\phi^2(u)))  \Big)
  +\Big(\mu(\beta(c),\{\alpha(a),c\})+ \lambda(\beta(b),\rho(\alpha(a),w))
   \nonumber\\
 &&  -\lambda(\beta(b),\rho(\beta\alpha^{-1}(c),\psi^{-1}\phi^2(u)))+\lambda(\{\beta(a),\alpha^{-1}\beta(c)\},\phi(v)) \Big)
 \mbox{ ( rearranging terms )}\nonumber\\
 &&=[\bar{\beta}(a+u),(b+v)]\ast\bar{\beta}(c+w)+\bar{\beta}(b+v)\ast[\bar{\alpha}(a+u),(c+w)]\nonumber
 \end{eqnarray}
 Hence, the conclusion follows.
\end{proof}
\begin{rmk} Consider the  split null extension $A\oplus V$ determined by the left BiHom-Poisson module $(V,\phi,\psi)$ for the BiHom-Poisson algebra
$(A, \{\cdot,\cdot\}, \mu, \alpha,\beta)$ in the previous theorem. Write  elements $a+v$ of $A\oplus V$ as $(a,v).$ Then there is an injective homomorphism of BiHom-modules
$i :V\rightarrow A\oplus V $ given by $i(v)=(0,v)$ and a surjective homomorphism of BiHom-modules $\pi : A\oplus V\rightarrow A$ given by $\pi(a,v)=a.$
Moreover, $i(V)$ is a two-sided BiHom-ideal of $A\oplus V$  such that $A\oplus V/i(V)\cong A$. On the other hand, there is a morphism of BiHom-algebras
$\sigma: A\rightarrow A\oplus V$ given by $\sigma(a)=(a,0)$ which is clearly a section of $\pi.$ Hence, we obtain the abelian split exact sequence of
BiHom-Poisson algebras and $(V, \phi,\psi)$ is a left BiHom-Poisson module for $A$ via $\pi.$
 \end{rmk}

\begin{defn}
Let $(A, \{\cdot,\cdot\}, \mu, \alpha,\beta)$ be a BiHom Poisson algebra. A skew-symmetric $n$-linear map $f: \underbrace{A\times \cdots \times A}_{n ~times}\rightarrow A$ that is a
derivation in each argument is called an $n$-BiHom-cochain, if it satisfies
$$\begin{array}{llllllllllll}
f(\alpha(x_1), \cdots, \alpha(x_n))&=&\alpha\circ f(x_{1}, \cdots, x_{n}),\label{cochain2}\\
f(\beta(x_1), \cdots, \beta(x_n))&=&\beta\circ f(x_{1}, \cdots, x_{n}).\label{cochain3}
\end{array}$$
The set of $n$-Hom-cochains is denoted by $C_{\alpha,\beta}^n(A,A)$, for $n\geq1$.
\end{defn}

\begin{defn}
Let $(A, \{\cdot,\cdot\}, \mu, \alpha,\beta)$ be a regular BiHom Poisson algebra. For $n=1,2,$ the coboundary operator $\delta^{n}:C_{\alpha,\beta}^n(A,A)\rightarrow C_{\alpha,\beta}^{n+1}(A,A)$ is defined as follows:
\begin{eqnarray}
&&\delta^{1}f(x,y)=\{\alpha(x),f(y)\}-\{f(x),\alpha(y)\}-f(\{\alpha^{-1}\beta(x),y\})\\
&&\delta^{2}f(x,y,z)=\{\alpha\beta(x),f(y,z)\}-\{\alpha\beta(y),f(x,z)\}+
\{\alpha\beta(y),f(x,z)\}
\nonumber\\
&&-f(\{\alpha^{-1}\beta(x),y\},\beta(z))+f(\{\alpha^{-1}\beta(x),z\},\beta(y))-f(\{\alpha^{-1}\beta(y),z\},\beta(x))
\end{eqnarray}
\end{defn}

\begin{lem}
The coboundary operators $\delta^i$ are well defined, for $i=1, 2$.
\end{lem}
\begin{proof}
For any $x,y,z\in A$ we have:
\begin{eqnarray}
&&\delta^{1}f(\alpha(x),\alpha(y))=\{\alpha^2(x),f\alpha(y)\}-\{f\alpha(x),\alpha^2(y)\}-f(\{\alpha^{-1}\beta\alpha(x),\alpha y)\})\nonumber\\
&&=\{\alpha^2(x),\alpha f(y)\}-\{\alpha f(x),\alpha^2(y)\}-f(\{\alpha\alpha^{-1}\beta(x),\alpha(y)\})=\alpha\circ\delta^1 f(x, y)\nonumber
\end{eqnarray}
and
\begin{eqnarray}
&&\delta^{2}f(\alpha(x),\alpha(y),\alpha(z))=\{\alpha\beta\alpha(x),f(\alpha(y),\alpha(z))\}-\{\alpha\beta\alpha(y),f(\alpha(x),\alpha(z))\}
\nonumber\\
&&+
\{\alpha\beta\alpha(y),f(\alpha(x),\alpha(z))\}-f(\{\alpha^{-1}\beta\alpha(x),\alpha(y)\},\beta\alpha(z))+f(\{\alpha^{-1}\beta\alpha(x),\alpha(z)\},\beta\alpha(y))\nonumber\\
&&-f(\{\alpha^{-1}\beta\alpha(y),\alpha(z)\},\beta\alpha(x))=
\{\alpha\beta\alpha(x),\alpha f(y,z)\}-\{\alpha\beta\alpha(y),\alpha f(x,z)\}
\nonumber\\
&&+
\{\alpha\beta\alpha(y),\alpha f(x,z)\}-f(\alpha(\{\alpha^{-1}\beta(x),y\}),\alpha\beta(z))+f(\alpha(\{\alpha^{-1}\beta()x,z\}),\alpha\beta(y))\nonumber\\
&&-f(\alpha(\{\alpha^{-1}\beta(y),z\}),\alpha\beta(x))=\alpha\circ\delta^2 f(x,y,z)\nonumber
\end{eqnarray}
In the same way, we obtain $\delta^{1}f\circ\beta=\beta\circ\delta^1f$ and $\delta^{2}f\circ\beta=\beta\circ\delta^2 f.$

Then $\delta^i$ are well defined, for $i=1,2$.
\end{proof}
\begin{thm}
With notations as above, we have
$$\delta^{2}\circ\delta^{1}=0.$$
\end{thm}
\begin{proof} Let $f\in C_{\alpha,\beta}^1(A,A)$ and $(x,y,z)\in A^{\times 3}$ then, we have:
\begin{eqnarray}
&&\delta^{2}\circ\delta^{1} f(x,y,z)=
\{\alpha\beta(x),\delta^{1} f(y,z)\}-\{\alpha\beta(y),\delta^{1} f(x,z)\}+
\{\alpha\beta(y),\delta^{1} f(x,z)\}
\nonumber\\
&&-\delta^{1} f(\{\alpha^{-1}\beta(x),y\},\beta(z))+\delta^{1} f(\{\alpha^{-1}\beta(x),z\},\beta(y))-\delta^{1} f(\{\alpha^{-1}\beta(y),z\},\beta(x))\nonumber\\
&=&\{\alpha\beta(x),\{\alpha(y),f(z)\}\}-\{\alpha\beta(x),\{\alpha(z),f(y)\}\}
-\{\alpha\beta(x),f(\{\alpha^{-1}\beta(y),z\})\}\nonumber\\
&&-\{\alpha\beta(y),\{\alpha(x),f(z)\}\}+\{\alpha\beta(y),\{\alpha(z),f(x)\}\}
+\{\alpha\beta(y),f(\{\alpha^{-1}\beta(x),z\})\}\nonumber\\
&&+\{\alpha\beta(z),\{\alpha(x),f(y)\}\}-\{\alpha\beta(z),\{\alpha(y),f(x)\}\}
-\{\alpha\beta(z),f(\{\alpha^{-1}\beta(x),y\})\}\nonumber\\
&&-\{\{\beta(x),\alpha(y)\}, f\beta(z)\}+\{\alpha\beta(z),f(\{\alpha^{-1}\beta(x),y\}),\}+f(\{\{\alpha^{-2}\beta^2(x),\alpha^{-1}\beta(y)\},\beta(z)\})\nonumber\\
&&+\{\{\beta(x),\alpha(z)\}, f\beta(y)\}-\{\alpha\beta(y),f(\{\alpha^{-1}\beta(x),z\}),\}-f(\{\{\alpha^{-2}\beta^2(x),\alpha^{-1}\beta(z)\},\beta(y)\})\nonumber\\
&&-\{\{\beta(y),\alpha(z)\}, f\beta(x)\}+\{\alpha\beta(x),f(\{\alpha^{-1}\beta(y),z\}),\}+f(\{\{\alpha^{-2}\beta^2(y),\alpha^{-1}\beta(z)\},\beta(x)\})\nonumber\\
&=&\{\beta^2(\alpha\beta^{-1}(x)),\{\beta(\alpha\beta^{-1}(y)),\alpha(\alpha^{-1}f(z))\}\}+\{\beta^2(\alpha\beta^{-1}(x)),\{\beta(\alpha^{-1} f(y)),\alpha(\beta\alpha^{-1}(z))\}\}\nonumber\\
&&-\{\alpha\beta(x),f(\{\alpha^{-1}\beta(y),z\})\}-\{\beta^2(\alpha\beta^{-1}(y)),\{\beta(\alpha^{-1}f(z)),\alpha(\beta\alpha^{-1}(x))\}\}\nonumber\\
&&+\{\beta^2(\alpha\beta^{-1}(y)),\{\beta(\beta^{-1}\alpha(z)),\alpha(\alpha^ {-1}f(x))\}\}
+\{\alpha\beta(y),f(\{\alpha^{-1}\beta(x),z\})\}\nonumber\\
&&+\{\beta^2(\alpha\beta^{-1}(z)),\{\beta(\beta^{-1}\alpha(x)),\alpha(\alpha^{-1}f(y))\}\}+\{\beta^2(\alpha\beta^{-1}(z)),\{\beta(\alpha^{-1} f(x)),\alpha(\alpha\beta^{-1}(y))\}\}\nonumber\\
&&-\{\alpha\beta(z),f(\{\alpha^{-1}\beta(x),y\})\}
+\{\beta^2(\alpha^{-1}f(z)),\{\beta(\alpha\beta^{-1}(x)),\alpha(\beta^{-1}\alpha(y))\}\}
\nonumber\\
&&+\{\alpha\beta(z),f(\{\alpha^{-1}\beta(x),y\}),\}-f(\{\beta^2(\alpha^{-1}(z)),\{\beta(\alpha^{-1}(x)),\alpha(\alpha^{-1}(y))\}\})\nonumber\\
&&+\{\beta^2(\alpha^{-1}f(y)),\{\beta(\beta^{-1}\alpha(z)),\alpha(\beta^{-1}\alpha(x))\}\}-\{\alpha\beta(y),f(\{\alpha^{-1}\beta(x),z\}),\}
\nonumber\\
&&-f(\{\beta^2(\alpha^{-1}(y)),\{\beta(\alpha^{-1}(z)),\alpha(\alpha^{-1}(x))\}\})+\{\beta^2(\alpha^{-1}f(x)),\{\beta(\alpha\beta^{-1}(y)),\alpha(\alpha\beta^{-1}(z))\}\}\nonumber\\
&&+\{\alpha\beta(x),f(\{\alpha^{-1}\beta(y),z\}),\}-f(\{\beta^2(\alpha^{-1}(x)),\{\beta(\alpha^{-1}(y)),\alpha(\alpha^{-1}(z))\}\})\nonumber\\
&& \mbox{ ( since  $\{u,v\}=-\{\beta\alpha^{-1}(v),\alpha\beta^{-1}(u)\}$ $\forall u,v\in A$)}\nonumber\\
&&=0 \mbox{ ( by  the BiHom-Jaobi identity) }. \nonumber
\end{eqnarray}
\end{proof}
For $n=1,2,$ the map $f\in C_{\alpha,\beta}^n(A,A)$ is called an $n$-BiHom-cocycle $\delta^n f=0.$ We denote the subspace spanned by $n$-Bihom-cocycles by
 $Z_{\alpha,\beta}^n(A,A)$ and $B_{\alpha,\beta}^n(A,A)=\delta^{n-1}C_{\alpha,\beta}^{n-1}(A,A).$ Since  $\delta^2\circ\delta^1=0,$ $B_{\alpha,\beta}^2(A,A)$
 is a subspace of $Z_{\alpha,\beta}^2(A,A).$ Hence we can define a cohomology space $H_{\alpha,\beta}^2(A,A)$ of as the factor space
 $Z_{\alpha,\beta}^2(A,A)/B_{\alpha,\beta}^2(A,A).$

\end{document}